\documentclass[final,leqno]{siamltex}

\newcommand{\bu}{\boldsymbol u}
\newcommand{\bx}{\boldsymbol x}
\newcommand{\by}{\boldsymbol y}
\newcommand{\bbR}{\mathbb R}

\usepackage{algorithm}% http://ctan.org/pkg/algorithms
\usepackage{algpseudocode}% http://ctan.org/pkg/algorithmicx
\usepackage{amsmath, amssymb}
\usepackage[usenames,dvipsnames,svgnames,table]{xcolor}
\usepackage{multirow}
\usepackage{url}
\usepackage{setspace}

\usepackage{tikz,pgfplots}
\usepackage{graphicx}
\usepackage{subfigure}
\usepackage{placeins}

\usepackage[left=1in, right=1in, top=1in, bottom=1in]{geometry}
\newtheorem{remark}[theorem]{Remark}

\usepackage[numbers,sort,compress]{natbib}

\title{Fast symmetric factorization of hierarchical matrices with
  applications}

\author{ Sivaram Ambikasaran\thanks{Department of Computational \& Data Sciences, Indian Institute of Science, \texttt{sivaram@cds.iisc.ac.in}}
  \and Michael O'Neil\thanks{Courant Institute of Mathematical Sciences and Tandon School of Engineering, New York University, \texttt{oneil@cims.nyu.edu}}
  \and Karan Raj Singh\footnotemark[1]}

\definecolor{mycyan}{rgb}{0,0.95,0.95}

\pgfplotsset{compat=newest}
\begin{document}

\maketitle

\begin{abstract}
We present a fast direct algorithm for computing symmetric
factorizations, i.e. $A = WW^T$, of symmetric positive-definite hierarchical 
matrices with weak-admissibility
conditions. The computational cost for the symmetric factorization
scales as $\mathcal{O}(n \log^2 n)$ for hierarchically off-diagonal
low-rank matrices. Once this factorization is obtained, the cost for
inversion, application, and determinant computation scales as
$\mathcal{O}(n \log n)$. In particular, this allows for the near
optimal generation of correlated random variates in the case where $A$
is a covariance matrix.  This symmetric factorization algorithm
depends on two key ingredients. First, we present a novel symmetric
factorization formula for low-rank updates to the identity of the form
$I+UKU^T$. This factorization can be computed in $\mathcal{O}(n)$ time
if the rank of the perturbation is sufficiently small.  Second,
combining this formula with a recursive divide-and-conquer strategy,
near linear complexity symmetric factorizations for hierarchically
structured matrices can be obtained. We present numerical results for
matrices relevant to problems in probability \& statistics (Gaussian
processes), interpolation (Radial basis functions), and Brownian dynamics
calculations in fluid mechanics (the Rotne-Prager-Yamakawa tensor). 
\end{abstract}

\begin{keywords}
Symmetric factorization, Hierarchical matrix, Fast algorithms,
Covariance matrices, Direct solvers, Low-rank, Gaussian processes,
Multivariate random variable generation, Mobility matrix,
Rotne-Prager-Yamakawa tensor.
\end{keywords}

\begin{AMS}
15A23, 15A15, 15A09
\end{AMS}

\pagestyle{myheadings}
\thispagestyle{plain}
\markboth{Sivaram Ambikasaran and Michael O'Neil}{Fast symmetric factorization}

\section{Introduction}
\label{section_Introduction}
This article describes a computationally efficient method for
constructing the symmetric factorization of large dense matrices. The
symmetric factorization of large dense matrices is important in
several fields, including, among others, data
analysis~\cite{cichocki2009nonnegative,pauca2006nonnegative,wang2008multi},
geostatistics~\cite{matheron1963principles,wackernagel2003multivariate},
and hydrodynamics~\cite{geyer2009n2,jiang2013fast}.  For instance,
several schemes for multi-dimensional Monte Carlo simulations require
drawing covariant realizations of multi-dimensional random variables.
In particular, in the case where the marginal distribution of each
random variable is normal, the covariant samples can be obtained by
applying the {\em symmetric factor} of the corresponding covariance
matrix to independent normal random vectors. The symmetric
factorization of a symmetric positive-definite matrix is given as
the factor $W$ in $A = W^T W$.  One of the major computational
issues in dealing with large covariance matrices is that they are
often dense. Conventional methods of obtaining a symmetric
factorization based on the Cholesky decomposition are expensive, since
the computational cost scales as $\mathcal{O}(n^3)$ for an $n \times
n$ matrix. Relatively recently, however, it has been observed that
large dense (full-rank) covariance matrices can be efficiently
represented using hierarchical
decompositions~\cite{ambikasaran2014fast, ambikasaran2013large,
  ambikasaran2013fastBayes, ambikasaran2012application, li2014kalman,
  chendata}. Taking advantage of this underlying structure, we derive
a novel symmetric factorization for large dense Hierarchical
Off-Diagonal Low-Rank (HODLR) matrices that scales as $\mathcal{O}(n
\log^2 n)$. That is to say, for a given $n\times n$ matrix $A$, we
decompose it as $A = WW^T$.  A major difference of our scheme versus
the Cholesky decomposition is the fact that the matrix $W$ is no
longer a triangular matrix. In fact, the algorithm of this paper
constructs the matrix $W$ as a product of
matrices which are block low-rank updates of the identity matrix. The
cost of applying the resulting factor $W$ to a vector scales as
$\mathcal{O}(n \log n)$.

Hierarchical matrices were first introduced in the context of integral
equations~\cite{hackbusch1999sparse,hackbusch2000sparse} arising out
of elliptic partial differential equations and potential theory. Since
then, it has been observed that large classes of dense matrices
arising out of boundary integral equations~\cite{ying2009fast,lai2014fast,ambikasaran2014ifmm}, dense
fill-ins in finite element
matrices~\cite{xia2009superfast,aminfar2014fast}, radial basis
function interpolation~\cite{ambikasaran2013fast}, kernel density
estimation in machine learning, covariance structure in statistical
models~\cite{chendata}, Bayesian inversion~\cite{ambikasaran2013fast,
  ambikasaran2013large, ambikasaran2013fastBayes}, Kalman
filtering~\cite{li2014kalman}, and Gaussian
processes~\cite{ambikasaran2014fast} can be efficiently represented as
data-sparse hierarchical matrices. After a suitable ordering of
columns and rows, these matrices can be recursively sub-divided using
a tree structure. Certain sub-matrices at each level in the tree
can then be well-represented by low-rank matrices. We refer the readers
to~\cite{hackbusch1999sparse, hackbusch2000sparse,
  grasedyck2003construction, hackbusch2002data, borm2003hierarchical,
  chandrasekaran2006fast, chandrasekaran2006fast1,
  ambikasaran2013thesis,ambikasaran2014ifmm} for more details on these matrices.
Depending on the tree structure and low-rank approximation technique,
different hierarchical decompositions exist. For example, the original
fast multipole method~\cite{greengard1987fast} (from now on abbreviated as FMM) accelerates the
calculation of long-range gravitational forces for $n$-body problems
by hierarchically compressing (via a quad- or oct-tree) certain
interactions in the associated matrix operator using analytical
low-rank considerations. The low-rank sparsity structure of these
hierarchical matrices can be exploited to construct fast dense linear
algebra schemes, including direct inversion, determinant computation,
symmetric factorization, etc. Broadly speaking, different hierarchical
matrices can be divided into categories based on two main criteria:
(i) Admissibility (strong and weak), (ii) Nested low-rank basis. The admissibility criterion identifies and specifies sub-blocks of the hierarchical matrix that can be represented as a low-rank matrix, while the nested low-rank basis enables additional compression of the low-rank sub-matrices, by assuming that the low-rank basis of a parent can be obtained using the low-rank basis of its children. (For instance, the difference between ``FMMs" and ``tree-codes'' is that
there are no translation operators between one level to the other in
the case of tree-codes). A detailed analysis of these different hierarchical structures is
provided in~\cite{ambikasaran2013thesis,ambikasaran2014ifmm}. We also
refer readers to~\cite{hackbusch2004hierarchical} for a thorough
discussion and analysis of hierarchical matrices with weak
admissibility criteria.

Most of the existing results relevant to the symmetric factorization of
low-rank modifications to the identity are based on rank $1$ or rank
$r$ modifications to the Cholesky factorization, which are
computationally expensive, i.e., their scaling is at least
$\mathcal{O}(rn^2)$. We do not seek to review the entire literature
here, except to direct the readers to a few references~\cite{seeger2007low, higham1986newton, gill1974methods,
  bjorck1983schur}.

The scheme presented here for HODLR matrices scales as $\mathcal{O}(n
\log^2 n)$. The algorithm also extends to other hierarchical
structures naturally. However, for matrices with strong admissibility
criteria, the algorithm scales depending on the underlying dimension
as illustrated in Section~\ref{section_numerical}. Xia and
Gu~\cite{xia2010robust} also discuss a Cholesky factorization for
hierarchical matrices (Hierarchically Semi-Separable matrices to be
specific, which are also termed as $\mathcal{H}^2$ matrices with weak
admissibility criteria~\cite{hackbusch2000h2}). Once the HSS
factorization is computed, their algorithm scales as $\mathcal{O}(n)$.
As presented, the cost of constructing the HSS representation in their
article scales as $\mathcal{O}(n^2)$. It is possible to reduce the
computational cost of forming an HSS representation to $\mathcal{O}(n)$,
for instance, using randomized algorithms as discussed
  in~\cite{martinsson2011fast}. The HODLR
  matrices have a weaker assumption on the off-diagonal blocks as
  opposed to HSS matrices and therefore algebraic techniques for
  assembling HODLR matrices are simpler and easier. A comparison of
  the computational cost for each step of our algorithm and the
  algorithm by Xia and Gu~\cite{xia2010robust} is presented in
  Table~\ref{table_computational_costs}.

\begin{table}[!htbp]
\begin{center}
\caption{Computational cost of each of the significant step in symmetric factorization}
\begin{tabular}{|c|c|c|c|c|c|}
\hline
& Matrix type & Algorithm & Assembly & Factorization & Solve\\
\hline
Our method & HODLR & Symmetric factorization & $\mathcal{O}(n \log n)$ & $\mathcal{O}(n \log^2 n)$ & $\mathcal{O}(n \log n)$\\
\hline
Xia and Gu~\cite{xia2010robust} & HSS & Cholesky & $\mathcal{O}(n^2)$ & $\mathcal{O}(n)$ & $\mathcal{O}(n)$\\
\hline
\end{tabular}
\end{center}
\label{table_computational_costs}
\end{table}

The paper is organized as follows:
Section~\ref{section_symfactor_lowrank} contains the key idea behind
the algorithm discussed in this paper: a fast, symmetric factorization
for low-rank updates to the identity.
Section~\ref{section_hierarchical} extends the formula of
Section~\ref{section_symfactor_lowrank} to a nested product of block
low-rank updates to the identity. The details of the compatibility of
this structure with HODLR matrices is
discussed. Section~\ref{section_numerical} contains numerical results
(accuracy and complexity scaling) of applying the factorization
algorithms to matrices relevant to problems in statistics,
interpolation and hydrodynamics. Section~\ref{section_conclusion}
summarizes the previous results and discusses further extensions and
areas of ongoing research. The algorithm discussed in this article has been implemented in C++ (parallelized using OpenMP) and the implementation is made available at \url{https://gitlab.com/SAFRAN/HODLR}.

\section{Acknowledgements} The research supported in part by the Air Force
    Office of Scientific Research under NSSEFF Program Award
    FA9550-10-1-0180. SA was also supported by the INSPIRE faculty award [DST/INSPIRE/04/2014/001809] by the Department of Science \& Technology, India and the startup grant provided by the Indian Institute of Science.
\section{Symmetric Factorization of low-rank update}
\label{section_symfactor_lowrank}

Almost all of the hierarchical factorizations are typically based on
incorporating low-rank perturbations in a hierarchical manner. In this
section, we briefly discuss some well-known identities which allow for
the rapid inversion and determinant computation of low-rank updates to
the identity matrix.

\subsection{The Sherman-Morrison-Woodbury formula}
If the inverse of a matrix $A \in \mathbb R^{n\times n}$ is already
known, then the inverse of subsequent low-rank updates, for $U, V \in
\mathbb R^{n\times p}$ and $C \in \bbR^{p \times p}$, can be
calculated as
\begin{align}
\left( A + UCV^T \right)^{-1} = A^{-1} - A^{-1} U \left( C^{-1} + V^T A^{-1}U
\right)^{-1} V^T A^{-1},
\end{align}
where we should point out that the quantity $V^TU$ is only a $p \times
p$ matrix. This formula is known as the Sherman-Morrison-Woodbury
(SMW) formula. Further simplifying, in the case where $A = I$, we have
\begin{align}
\label{equation_Sherman_Morrison_Woodbury}
\left(I + U C V^T \right)^{-1} = I - U (C^{-1} + V^T U)^{-1}V^T.
\end{align}
Note that the SMW formula shows that the inverse of a low-rank
perturbation to the identity matrix is also a low-rank perturbation to
the identity matrix. Furthermore, the row-space and column-space of
the low-rank perturbation and its inverse are the same. The main
advantage of Equation~\eqref{equation_Sherman_Morrison_Woodbury} is
that if $p \ll n$, we can obtain the inverse (or equivalently solve a
linear system) of a rank $p$ perturbation of an $n \times n$ identity
matrix at a computational cost of $\mathcal{O}(p^2n)$. In general, if
$B \in \mathbb{R}^{n \times n}$ is a low-rank perturbation of $A \in
\mathbb{R}^{n \times n}$, then the inverse of $B$ is also a
low-rank perturbation of the inverse of $A$.

It is also worth noting that if $A$ and $B$ are well-conditioned, then
the Sherman-Morrison-Woodbury formula is numerically
stable~\cite{yip1986note}. The SMW formula has found applications in,
to name a few, Kalman filters~\cite{mandel2006efficient}, recursive
least-squares~\cite{geist2010statistically}, and fast direct solvers
for hierarchical matrices~\cite{ambikasaran2013fast, starr1991numerical, lyons2005fast}.

\subsection{Sylvester's determinant theorem}
Calculating the determinant of an $n \times n$ matrix $A$,
classically, using a cofactor expansions requires $\mathcal O(n!)$
operations. However, this can be reduced to $\mathcal O(n^3)$
by first computing the $LU$ or eigenvalue decomposition of the
matrix.
Recently~\cite{ambikasaran2014fast}, it was shown that the
determinant of HODLR matrices could be calculated in $\mathcal O(n
\log n)$ time using Sylvester's determinant
theorem~\cite{akritas1996various}, a formula relating the determinant
of a low-rank update of the identity to the determinant of a smaller
matrix. Determinants of matrices are very important in probability and
statistics, in particular in Bayesian inference, as they often serve
as the normalizing factor in likelihood calculations and in the
evaluation of the conditional evidence.

Sylvester's determinant theorem states that for $A,B \in \mathbb R^{n
  \times p}$,
\begin{align}
\label{equation_Sylvester_Determinant}
\det \left(I + AB^T \right) = \det \left(I + B^T A \right),
\end{align}
where the determinant on the right hand side is only of a $p \times p$
matrix.  Hence, the determinant of a rank $p$ perturbation to an $n
\times n$ identity matrix, where $p \ll n$, can be computed at a
computational cost of $\mathcal{O}(p^2n)$. This formula has recently
found applications in Bayesian statistics for computing precise values
of Gaussian likelihood functions (which depend on the determinant of
the corresponding covariance matrix)~\cite{ambikasaran2014fast} and
computing the determinant of large matrices in random matrix
theory~\cite{tao2012topics}.

\subsection{Symmetric factorization of a low-rank update}

In the spirit of the Sherman-Morrison-Woodbury formula and Sylvester's
determinant theorem, we now obtain a formula that enables the symmetric
factorization of a rank $p$ perturbation to the $n \times n$ identity
at a computational cost of $\mathcal{O}(p^2n)$. In particular, for a
symmetric positive-definite (from now on abbreviated as SPD) matrix of
the form $I + UKU^T$, where $I$ is an $n \times n$ identity matrix, $U
\in \mathbb{R}^{n \times p}$, $K \in \mathbb{R}^{p \times p}$, and $p
\ll n$, we obtain the factorization
\begin{align}
I + UKU^T = W W^T.
\end{align}
We now state this as the following theorem.

\begin{theorem}
\label{thm_main}
For rank $p$ matrices $U \in \bbR^{n \times p}$ and $K \in \bbR^{p \times p}$, if the matrix $I + UKU^T$ is SPD then it can be symmetrically factored as
\begin{align}
\label{eq_main}
I+UKU^T = \left(I+UXU^T\right) \left(I+UXU^T\right)^T
\end{align}
where $X$ is obtained as
\begin{align}
X &= L^{-T} \left( M - I \right) L^{-1},
\label{equation_X}
\end{align}
the matrix $L$ is the symmetric factor of $U^TU$, and $M$ is the symmetric
factor of $I+L^TKL$, i.e.,
\begin{align}
LL^T &= U^T U,\\
MM^T &= I + L^T K L.
\label{equation_MMT}
\end{align}
\end{theorem}
We first prove two lemmas related to the construction of $X$ in
Equation~\eqref{equation_X}, which directly lead to the proof of
Theorem~\ref{thm_main}. In the subsequent discussion, we will assume
the following unless otherwise stated:
\begin{enumerate}
\item $I$ is the identity matrix.
\item $p \ll n$.
\item $U \in \bbR^{n \times p}$ is of rank $p$.
\item $K \in \mathbb{R}^{p \times p}$ is of rank $p$.
\item $I + UKU^T$ is SPD.
\end{enumerate}
It is easy to show that the last item implies that the matrix $K$ is
symmetric.  The first lemma we prove relates the positivity of the
smaller matrix $I + L^TKL \in \bbR^{p \times p}$ to the positivity of
the larger $n \times n$ matrix, $I + UKU^T$.
%\begin{proof} Since $I_N + UKU^T$ is symmetric, we have $(I_N +
%UKU^T)^T = I_N + UKU^T$. This gives us $I_N + U K^T U^T = I_N +
%UKU^T$, which in turn gives $U(K^T-K)U^T = 0$. Since $U$ is a thin
%full-rank matrix, we can conclude that $K^T = K$ and hence $K$ is
%symmetric.  \end{proof}

\begin{lemma}
\label{prop_2}
Let $LL^T$ denote a symmetric factorization of $U^TU$, where $L \in
\mathbb{R}^{p \times p}$. If the matrix $I+UKU^T$ is SPD
(semi-definite), then $I+L^TKL$ is also SPD (semi-definite).
\end{lemma}
\begin{proof}
To prove that $I+L^TKL$ is SPD, it suffices to prove that given any
non-zero $\by \in \mathbb{R}^p$, we have $\by^T(I+L^TKL)\by > 0$. Note
that since $U$ is full rank, the matrix $U^TU$ is invertible. We now
show that given any $\by \in \mathbb{R}^p$, there exists an $\bx \in
\mathbb{R}^n$ such that $\by^T (I + L^TKL) \by = \bx^T (I + UKU^T)
\bx$. This will enable us to conclude that $I + L^TKL$ is positive
definite since $I + UKU^T$ is positive-definite.  In fact, we will
directly construct $\bx \in \mathbb{R}^n$ such that $\by^T (I +
L^TKL) \by = \bx^T (I + UKU^T) \bx$.

Let us begin by choosing $\bx = U(U^TU)^{-1}L \by$. Then, the
following two criteria are met:
\begin{enumerate}
\item[(i)] $\Vert \bx \Vert_2 = \Vert \by \Vert_2$.
\item[(ii)] $\bx^T UKU^T\bx = \by^T L^T K L \by$.
\end{enumerate}
Expanding the norm of $\bx$ we have:
\begin{align}
\Vert \bx \Vert_2^2 &= \by^T L^T 
(U^TU)^{-T}\overbrace{U^TU(U^TU)^{-1}}^{I}L\by \\
&= \by^T L^T (U^TU)^{-T}L\by\\
& = \by^T L^T (LL^T)^{-T}L\by \\
&= \by^T L^T (L^{-T}L^{-1})L\by \\
&= \Vert \by \Vert_2^2.
\end{align}
This proves criteria (i).  Furthermore, by our choice of $\bx$, we
also have that $U^T\bx = L\by$. Therefore,
\begin{align}
\bx^TUKU^T\bx &= (U^T\bx)^T K (U^T\bx) \\
& = (L\by)^T K (L\by) \\
&= \by^TL^T K L\by.
\end{align}
This proves criteria (ii).
From the above, we can now conclude that
\begin{align}
\by^T (I + L^TKL) \by = \bx^T (I + UKU^T) \bx.
\end{align}
Hence, if $I+UKU^T$ is SPD, so is $I+L^TKL$. An identical calculation
proves the positive semi-definite case.
\end{proof}

We now state and prove a lemma required for solving a quadratic matrix
equation that arises in the subsequent factorization scheme.

\begin{lemma}
\label{prop_3}
A solution $X$ to the quadratic matrix equation
\begin{align}
\label{equation_Quadratic_Equation}
XLL^TX^T + X + X^T = K
\end{align}
with $L,K \in \mathbb{R}^{p \times p}$ 
and $L$ a full rank matrix is given by
\begin{align}
\label{equation_Quadratic_Solution}
X = L^{-T} (M-I) L^{-1},
\end{align}
where $M \in \mathbb{R}^{p \times p}$ is a symmetric factorization of
$I+L^T K L$, that is, $MM^T = I+L^T K L$.
\end{lemma}
\begin{proof}
First note that from Lemma~\ref{prop_2}, since $I+L^T K L$ is
positive-definite, the symmetric factorization $I+L^T K L = MM^T$
exists. Now the easiest way to check if
Equation~\eqref{equation_Quadratic_Solution} satisfies
Equation~\eqref{equation_Quadratic_Equation} is to plug in the value
of $X$ from Equation~\eqref{equation_Quadratic_Solution} in
Equation~\eqref{equation_Quadratic_Equation}. This yields:
\begin{align}
XLL^TX^T & = L^{-T} (M-I)
 \overbrace{L^{-1}L}^{I}L^T(L^{-T} (M-I) L^{-1})^T\\
& = L^{-T} (M-I) \underbrace{L^T L^{-T}}_{I} (M-I)^T L^{-1}.
\end{align}
Further simplifying the expression, we have:
\begin{align}
XLL^TX^T& = L^{-T} (M-I) (M^T-I) L^{-1}\\
& = L^{-T} (MM^T - M - M^T + I) L^{-1}\\
& = L^{-T} (2I - M - M^T + L^TKL) L^{-1}\\
& = L^{-T}(I-M)L^{-1} + L^{-T}(I-M^T)L^{-1} + K\\
& = -X - X^T + K.
\end{align}
Therefore, we have that $XLL^TX^T + X+X^T=K$
\end{proof}

% It is
% important to note that the symmetric factorization based on low-rank
% updates to Cholesky decomposition for the above matrix scales as
% $\mathcal{O}(pn^2)$. Further, the cost of storing and applying the
% factor obtained from the Cholesky decomposition to a vector scales
% as $\mathcal{O}(n^2)$. %It is to be noted that the factorization
% presented in this article is one such factorization and by no means
% the only way to obtain the symmetric factorization of these matrices
% at a computational cost of $\mathcal{O}(p^2n)$.  It is to be noted
% that the Cholesky update, even though is expensive, might be more
% stable in certain cases. We do not compare the stability of the
% method discussed in this article with the Cholesky update, though we
% do provide numerical benchmarks and comparison with the Cholesky
% update.

We are now ready to prove the main result, Theorem~\ref{thm_main}.

\begin{proof} (Proof of Theorem~\ref{thm_main} )
The proof follows immediately from the previous two lemmas.  With $X$
and $L$ as previously defined as in Equation~\eqref{equation_X}, we have
\begin{align}
(I + UXU^T)(I + UXU^T)^T &= I + U(X+X^T+XU^TUX^T)U^T \\
&= I + U(X+X^T+XLL^TX^T)U^T.
\end{align}
Since $X = L^{-T} (M-I) L^{-1}$ and $MM^T = I+L^TKL$, from
Lemma~\ref{prop_3} we have that $X+X^T+XLL^TX^T = K$. Substituting in
the previous equation, we get
\begin{align}
(I + UXU^T)(I + UXU^T)^T = I+UKU^T.
\end{align}
This proves the symmetric factorization.
\end{proof}

\begin{remark}
A slightly more numerically stable variant of factorization~\eqref{eq_main} is:
\begin{align}
\label{equation_Obtained_Stable_Factorization}
I + U K U^T = \left(I + Q \tilde X Q^T\right)
\left(I + Q \tilde X Q^T\right)^T,
\end{align}
where $Q$ is a unitary matrix such that $U=QR$.
\end{remark}

\begin{remark}
Even though the previous theorem only addresses the symmetric
factorization problem with no restrictions of the factors being symmetric (i.e., $W=W^T$), we can also easily obtain a {\em square-root factorization}
in a similar manner.  By this we mean that for a given symmetric
positive-definite matrix $A$, one can obtain a symmetric matrix $G$
such that $G^2 = A$. The key ingredient is obtaining a square-root
factorization of a low-rank update to the identity:
\begin{align}
I + UKU^T = (I+UXU^T)^2,
\end{align}
where $X$ is a symmetric matrix and satisfies
\begin{align}
K = X(UU^T) X + 2X.
\label{equation_squareroot}
\end{align}
The solution to Equation~\eqref{equation_squareroot} is given by
\begin{align}
X=L^{-1} (M-I) L^{-1},
\end{align}
where $L$ and $M$ are symmetric square-roots
of $U^TU$ and $I+LKL$:
\begin{align}
L^2 &= U^TU. \\
M^2 &= I+LKL.
\end{align}
These factorizations can easily be obtained via a singular value or
eigenvalue decomposition.  This can then be combined with the recursive
divide-and-conquer strategy discussed in the next section to yield an $\mathcal{O}(n \log^2 n)$ algorithm for computing square-roots of HODLR matrices.
\end{remark}

Theorem~\ref{thm_main} has the two following useful corollaries.

\begin{corollary}
\label{cor_1}
If $p=1$, i.e., the perturbation to the identity in Equation~\eqref{eq_main}
is of rank $1$, then 
\begin{align}
I + \bu\bu^T = (I+\alpha \bu\bu^T)(I+\alpha \bu\bu^T),
\end{align}
where $\alpha = \dfrac{\sqrt{1+\Vert \bu \Vert_2^2}-1}{\Vert \bu
  \Vert_2^2}$.
\end{corollary}

This result can also be found
in~\cite{gill1974methods}. Corollary~\ref{cor_2} extends low-rank
updates to SPD matrices {\em other} than the identity.

\begin{corollary}
\label{cor_2}
Given a symmetric factorization of the form $WW^T$, where the inverse
of $W$ can be applied fast (i.e., the linear system $W \bx = \by$ can
be solved fast), then a symmetric factorization of a SPD matrix of the
form $WW^T + UKU^T$, where $U \in \bbR^{n\times p}$ and $p \ll n$, can
also be obtained fast. For instance, if the linear system $W \bx= \by$
can be solved at a computational cost of $\mathcal{O}(n)$, then the
symmetric factorization $W(I+\tilde{U}X \tilde{U}^T)(I+\tilde{U}X
\tilde{U}^T)^T W^T$ can also be obtained at a computational cost of
$\mathcal{O}(n)$.
\end{corollary}

A numerical example demonstrating Corollary~\ref{cor_2} is contained
in Section~\ref{sec-fast}.

Note that the factorizations in Equations~\eqref{eq_main}
and~\eqref{equation_Obtained_Stable_Factorization} are similar to the
Sherman-Morrison-Woodbury formula; in each case, the symmetric factor
is a low-rank perturbation to the identity. Furthermore, the row-space
and column-space of the perturbed matrix are the same as the row-space
and column-space of the symmetric factors. Another advantage of the
factorization in Equations~\eqref{eq_main}
and~\eqref{equation_Obtained_Stable_Factorization} is that the storage
cost and the computational cost of applying the factor to a vector,
which is of significant interest as indicated in the introduction,
scales as $\mathcal{O}(pn)$.

We now describe a computational algorithm for finding the symmetric
factorization described in Theorem~\ref{thm_main}.
Algorithm~\ref{algorithm_main} lists the individual steps in computing
the symmetric factorization and their associated computational
cost. The only computational cost is in computing the matrix $X$
in Equation~\eqref{eq_main}. Note that the dominant cost is
the matrix-matrix product of an $p \times n$ matrix with a $n \times
p$ matrix.  The rest of the steps are performed on a lower
$p$-dimensional space.

%\subsection{Computational cost}
\begin{algorithm}[!htbp]
\begin{center}
\caption{Symmetric factorization of $I+UKU^T$.}
\label{algorithm_main}
\begin{tabular}[!htbp]{|c|l|c|}
\hline
Step & Computation & Cost\\
\hline
$1$ & Calculate $A = U^TU \in \mathbb{R}^{p \times p}$ & $\mathcal{O}(p^2n)$\\
\hline
$2$ & Factor $A$ as $LL^T$, where $L \in \mathbb{R}^{p \times p}$ & $\mathcal{O}(p^3)$\\
\hline
$3$ & Calculate $T = I+L^TKL \in \mathbb{R}^{p \times p}$ & $\mathcal{O}(p^3)$\\
\hline
$4$ & Factorize $T$ as $MM^T$, where $M \in \mathbb{R}^{p \times p}$ & $\mathcal{O}(p^3)$\\
\hline
$5$ & Calculate $X = L^{-T}(M-I_p)L^{-1} \in \mathbb{R}^{p \times p}$ & $\mathcal{O}(p^3)$\\
\hline
\end{tabular}
\end{center}
\end{algorithm}

\FloatBarrier

Algorithm~\ref{algorithm_main} can be made more stable by first
performing a $QR$ decomposition of the matrix $U$ as indicated in
Equation~\eqref{equation_Obtained_Stable_Factorization}. Here, the
dominant cost is in obtaining the $QR$ factorization of the matrix
$U$.
%where $U=QR$ (the $QR$ decomposition of $U$), $X = M-I_p \in \mathbb{R}^{p \times p}$ and $M$ is the Cholesky factor of $I+RKR^T$, i.e., $I+RKR^T = MM^T$.
This is described in Algorithm~\ref{algorithm_stable}.
\begin{algorithm}[!htbp]
\begin{center}
\caption{Stable version of symmetric factorization of $I+UKU^T$.}
\label{algorithm_stable}
\begin{tabular}[!htbp]{|c|l|c|}
\hline
Step & Computation & Cost\\
\hline
$1$ & $QR$ decompose $U=QR$, where $Q \in \mathbb{R}^{n \times p}$ and
 $R \in \mathbb{R}^{p \times p}$ & $\mathcal{O}(p^2n)$\\
\hline
$2$ & Calculate $T = I+RKR^T \in \mathbb{R}^{p \times p}$ & $\mathcal{O}(p^3)$\\
\hline
$3$ & Factorize $T$ as $MM^T$, where $M \in \mathbb{R}^{p \times p}$ & $\mathcal{O}(p^3)$\\
\hline
$4$ & Calculate $X = M-I \in \mathbb{R}^{p \times p}$ & $\mathcal{O}(p)$\\
\hline
\end{tabular}
\end{center}
\end{algorithm}

\begin{remark}
If $U$ is rank deficient, then the 
reduced-rank $QR$ decomposition must be computed. Algorithm 2
proceeds accordingly with a smaller internal rank $\tilde p$.
\end{remark}

\begin{remark}
Note that if $U$ is unitary, then Equation~\eqref{equation_X} reduces to $X = M-I$ and Equation~\eqref{equation_MMT} reduces to $MM^T = I + K$. Note that in our case $K$ is of the form
$$K = \begin{bmatrix}
0 & \tilde{K}^T\\
\tilde{K} & 0
\end{bmatrix}$$
Hence, the Cholesky decomposition of $I+K$ is of the form
$$I +K = \begin{bmatrix}
I & 0\\
\tilde{K} & D
\end{bmatrix} \begin{bmatrix}
I & \tilde{K}^T\\
0 & D^T
\end{bmatrix}$$
where we need $DD^T = I-\tilde{K}\tilde{K}^T$. Hence, to find the Cholesky decomposition of $I+K$, it suffices to find the Cholesky decomposition of $I-\tilde{K}\tilde{K}^T$. Note that matrix $I-\tilde{K}\tilde{K}^T$ is half the size of the matrix $I+K$.
\end{remark}

\begin{remark}
	From the above, we also note that if $U$ is unitary, we then have
	\begin{align}
		\det(M) & = \det(D)
		\label{equation_D} 
	\end{align}
	This inturn gives us that
	\begin{align}
		\det(I+UKU^T) & = \det(I+UXU^T)^2 = \det(I+X)^2 = \det(M)^2 = \det(D)^2
	\end{align}

\end{remark}

\section{Symmetric Factorization of HODLR matrices}
\label{section_hierarchical}
In this section, we extend the symmetric factorization formula of the
previous section to a class of hierarchically structured matrices,
known as Hierarchical Off-Diagonal Low-Rank matrices.

\subsection{Hierarchical Off-Diagonal Low Rank matrices}
There exists a variety of hierarchically structured matrices depending
on the particular matrix, choice of recursive sub-division (inherent
tree structure), and low-rank sub-block compression technique (refer
to Chapter~3 of~\cite{ambikasaran2013thesis} for more details). In
this article, we will focus on a specific class of matrices known as
Hierarchical Off-Diagonal Low-Rank (HODLR) matrices. Furthermore,
since this article deals with SPD matrices, we shall restrict our
attention to SPD-HODLR matrices. We first briefly review
the HODLR matrix structure.

A matrix $A \in \mathbb{R}^{n \times n}$ is termed a $2$-level HODLR
matrix if it can be written in the form:
\begin{align}
A & =
\begin{bmatrix}
A_1^{(1)} & U_{1}^{(1)} K_{1,2}^{(1)}V_{2}^{(1)^T}\\
U_{2}^{(1)} K_{2,1}^{(1)} V_{1}^{(1)^T} & A_2^{(1)}
\end{bmatrix}\\
& =
\begin{bmatrix}
\begin{bmatrix}
A_{1}^{(2)} & U_{1}^{(2)} K_{1,2}^{(2)} V_{2}^{(2)^T}\\
U_{2}^{(2)} K_{2,1}^{(2)} V_{1}^{(2)^T} & A_{2}^{(2)}
\end{bmatrix}
&
U_{1}^{(1)} K_{1,2}^{(1)} V_{2}^{(1)^T}\\
U_{2}^{(1)} K_{2,1}^{(1)} V_{1}^{(1)^T}&
\begin{bmatrix}
A_{3}^{(2)} & U_{3}^{(2)} K_{3,4}^{(2)} V_{4}^{(2)^T}\\
U_{4}^{(2)} K_{4,3}^{(2)} V_{3}^{(2)^T} & A_{4}^{(2)}
\end{bmatrix}
\end{bmatrix},
\label{equation_HODLR_2_level}
\end{align}
where the off-diagonal blocks are of low-rank. Throughout this
section, for pedagogical purposes, we shall assume that the rank of
these off-diagonal blocks are the same at all levels. However, in practice
the ranks of different blocks on different levels is determined numerically,
and are frequently non-constant.

In general, for a $\kappa$-level HODLR matrix $A$, the $i^{th}$
diagonal block at level $k$, where $1 \leq i \leq 2^{k}$ and $0 \leq k
< \kappa$, denoted as $A_{i}^{(k)}$, can be written as
\begin{align}
A_{i}^{(k)} =
\begin{bmatrix}
A_{2i-1}^{(k+1)} & U_{2i-1}^{(k+1)} K_{2i-1,2i}^{(k+1)}V_{2i}^{(k+1)^T}\\
U_{2i}^{(k+1)} K_{2i, 2i-1}^{(k+1)} V_{2i-1}^{(k+1)^T} & A_{2i}^{(k+1)}
\end{bmatrix}
\label{equation_HODLR_level_k}
\end{align}
where 
\begin{align}
A_{i}^{(k)} &\in \mathbb{R}^{n/2^k \times n /2^k},\\
U_{2i-1}^{(k)}, U_{2i}^{(k)}, V_{2i-1}^{(k)}, V_{2i}^{(k)} &\in
\mathbb{R}^{n/2^k \times p},\\
K_{2i-1,2i}^{(k)}, K_{2i,2i-1}^{(k)} &\in \mathbb{R}^{p \times p},
\end{align}
and $p \ll n$. The maximum number of levels, $\kappa$, is $\big
\lfloor \log_2 ( n/2p ) \big \rfloor$.  Figure~\ref{figure_HODLR}
depicts the typical structure of a HODLR matrix at different levels.
See~\cite{ambikasaran2013fast} for details on the actual numerical
decomposition of a matrix into HODLR form.

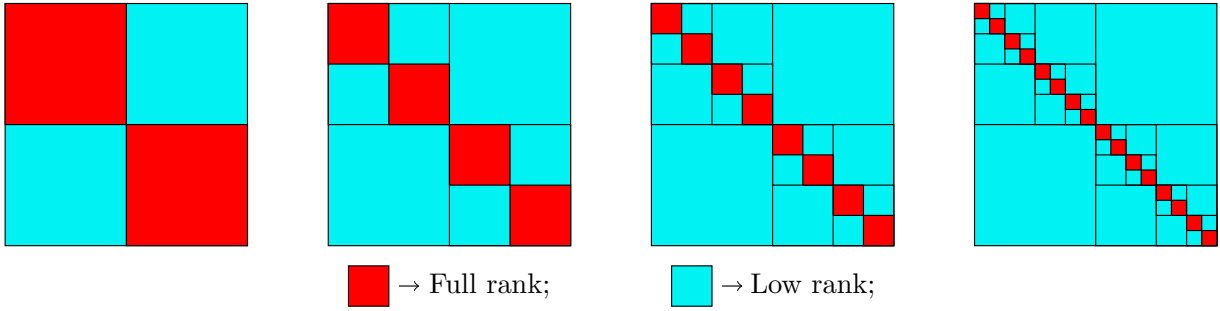
\begin{figure}[!tbp]
\resizebox{\hsize}{!}{
\begin{tikzpicture}[scale=1]
\draw[fill=white,draw = white] (-0.0625,-0.0625) rectangle (15.0625,3.0625);

%	Level 1
\draw[draw=black, fill=mycyan] (0,0) rectangle (3,3);
\foreach \i in {0, 1.5} {
	\draw[draw=black, fill=red] (\i, 3-\i) rectangle (\i+1.5, 1.5-\i);
}

%	Level 2
\draw[draw=black, fill=mycyan] (4,0) rectangle (7,3);
\foreach \i in {0, 1.5} {
	\draw[draw=black] (4+\i, 3-\i) rectangle (\i+5.5, 1.5-\i);
}
\foreach \i in {0, 0.75, 1.5, 2.25} {
	\draw[draw=black, fill=red] (4+\i, 3-\i) rectangle (\i+4.75, 2.25-\i);
}

%	Level 3
\draw[draw=black, fill=mycyan] (8,0) rectangle (11,3);
\foreach \i in {0, 1.5} {
	\draw[draw=black] (8+\i, 3-\i) rectangle (\i+9.5, 1.5-\i);
}
\foreach \i in {0, 0.75, 1.5, 2.25} {
	\draw[draw=black] (8+\i, 3-\i) rectangle (\i+8.75, 2.25-\i);
}
\foreach \i in {0, 0.375, 0.75, 1.125, 1.5, 1.875, 2.25, 2.625} {
	\draw[draw=black, fill=red] (8+\i, 3-\i) rectangle (\i+8.375, 2.625-\i);
}

%	Level 4
\draw[draw=black, fill=mycyan] (12,0) rectangle (15,3);
\foreach \i in {0, 1.5} {
	\draw[draw=black] (12+\i, 3-\i) rectangle (\i+13.5, 1.5-\i);
}
\foreach \i in {0, 0.75, 1.5, 2.25} {
	\draw[draw=black] (12+\i, 3-\i) rectangle (\i+12.75, 2.25-\i);
}
\foreach \i in {0, 0.375, 0.75, 1.125, 1.5, 1.875, 2.25, 2.625} {
	\draw[draw=black] (12+\i, 3-\i) rectangle (\i+12.375, 2.625-\i);
}

\foreach \i in {0, 0.1875, 0.375, 0.5625, 0.75, 0.9375, 1.125, 1.3125, 1.5, 1.6875, 1.875, 2.0625, 2.25, 2.4375, 2.625, 2.8125} {
	\draw[draw=black, fill=red] (12+\i, 3-\i) rectangle (\i+12.1875, 2.8125-\i);
}

%	Indicators
\def\x{4.25}
\draw[fill = red] (\x,-0.75) rectangle (\x+0.5,-0.25);
\draw[->] (\x+0.625,-0.5) -- (\x+0.875, -0.5);
\node at (\x+1.75,-0.5) {Full rank;};

\draw[fill = mycyan] (\x+4,-0.75) rectangle (\x+4.5,-0.25);
\draw[->] (\x+4.625,-0.5) -- (\x+4.875, -0.5);
\node at (\x+5.75,-0.5) {Low rank;};

\end{tikzpicture}
}
\caption{A HODLR matrix at different levels.}
\label{figure_HODLR}
\end{figure}

%%%%\FloatBarrier

\subsection{The symmetric factorization of a SPD-HODLR matrix}

Given a SPD-HODLR matrix $A\in \mathbb{R}^{n \times n}$, we wish to
obtain a symmetric factorization into $2\kappa+2$ block diagonal
matrices, that is we will factor $A$ as:
\begin{align}
A = \underbrace{A_{\kappa} A_{\kappa-1} A_{\kappa - 2} \cdots A_0}_W \overbrace{A_0^T \cdots A_{\kappa - 2}^T A_{\kappa-1}^T A_{\kappa}^T}^{W^T},
\label{equation_main_factorization}
\end{align}
where $A_k \in \mathbb{R}^{n \times n}$ is a block diagonal matrix
with $2^k$ diagonal blocks each of size $\dfrac{n}{2^k} \times
\dfrac{n}{2^k}$. The important feature of this factorization is that
each of the diagonal blocks on all levels is a low-rank update to
an identity matrix. Also, since $A$ is symmetric, 
we will assume that $U_{i}^{(k)} = V_{i}^{(k)}$ and
$K_{i,j}^{(k)} = K_{j,i}^{(k)^T}$ in Equation~\eqref{equation_HODLR_level_k}
for the remainder of the article. A graphical description of the $W$ factor
for a three-level SPD-HODLR
matrix is shown in Figure~\ref{figure_W_factor}.

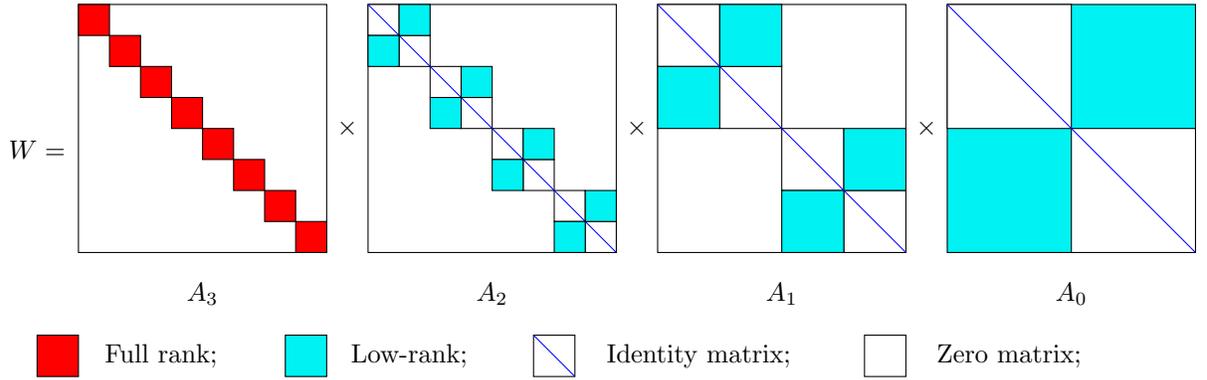
\begin{figure}[!tbp]
\begin{center}
\begin{tikzpicture}[scale=1.1]
\node at (3,1.25) {$W = $};

\draw[fill=white] (3.5,0) rectangle (6.5,3);
\draw[fill=red] (3.5,2.625) rectangle (3.875,3);
\draw[fill=red] (3.875,2.25) rectangle (4.25,2.625);
\draw[fill=red] (4.25,1.875) rectangle (4.625,2.25);
\draw[fill=red] (4.625,1.5) rectangle (5,1.875);
\draw[fill=red] (5,1.125) rectangle (5.375,1.5);
\draw[fill=red] (5.375,0.75) rectangle (5.75,1.125);
\draw[fill=red] (5.75,0.375) rectangle (6.125,0.75);
\draw[fill=red] (6.125,0) rectangle (6.5,0.375);

\node at (5,-0.5) {$A_3$};

\node at (6.75,1.5) {$\times$};

\draw[fill=white] (7,0) rectangle (10,3);
\draw[fill=white] (7,2.625) rectangle (7.375,3);
\draw[fill=white] (7.375,2.25) rectangle (7.75,2.625);
\draw[fill=white] (7.75,1.875) rectangle (8.125,2.25);
\draw[fill=white] (8.125,1.5) rectangle (8.5,1.875);
\draw[fill=white] (8.5,1.125) rectangle (8.875,1.5);
\draw[fill=white] (8.875,0.75) rectangle (9.25,1.125);
\draw[fill=white] (9.25,0.375) rectangle (9.625,0.75);
\draw[fill=white] (9.625,0) rectangle (10,0.375);

\draw[blue] (7,3) -- (10,0);

\draw[fill=mycyan] (7,2.25) rectangle (7.375,2.625);
\draw[fill=mycyan] (7.375,2.625) rectangle (7.75,3);
\draw[fill=mycyan] (7.75,1.5) rectangle (8.125,1.875);
\draw[fill=mycyan] (8.125,1.875) rectangle (8.5,2.25);
\draw[fill=mycyan] (8.5,0.75) rectangle (8.875,1.125);
\draw[fill=mycyan] (8.875,1.125) rectangle (9.25,1.5);
\draw[fill=mycyan] (9.25,0) rectangle (9.625,0.375);
\draw[fill=mycyan] (9.625,0.375) rectangle (10,0.75);

\node at (8.5,-0.5) {$A_2$};

\node at (10.25,1.5) {$\times$};

\draw[fill=white] (10.5,0) rectangle (13.5,3);
\draw[fill=white] (10.5,2.25) rectangle (11.25,3);
\draw[fill=white] (11.25,1.5) rectangle (12.0,2.25);
\draw[fill=white] (12,0.75) rectangle (12.75,1.5);
\draw[fill=white] (12.75,0) rectangle (13.5,0.75);
\draw[fill=mycyan] (10.5,1.5) rectangle (11.25,2.25);
\draw[fill=mycyan] (11.25,2.25) rectangle (12.0,3);
\draw[fill=mycyan] (12,0) rectangle (12.75,0.75);
\draw[fill=mycyan] (12.75,0.75) rectangle (13.5,1.5);

\draw[blue] (10.5,3) -- (13.5,0);

\node at (12,-0.5) {$A_1$};

\node at (13.75,1.5) {$\times$};

\draw[fill=mycyan] (14,0) rectangle (17,3);
\draw[fill=white] (14,1.5) rectangle (15.5,3);
\draw[fill=white] (15.5,0) rectangle (17,1.5);

\draw[blue] (14,3) -- (17,0);

\node at (15.5,-0.5) {$A_0$};

\draw[fill = red] (-1+4,-1.5) rectangle (-0.5+4,-1);
\node at (0.5+4,-1.25) {Full rank;};

\draw[fill = mycyan] (2+4,-1.5) rectangle (2.5+4,-1);
\node at (3.5+4,-1.25) {Low-rank;};

\draw[fill = white] (5+4,-1.5) rectangle (5.5+4,-1);
\draw[blue] (5+4,-1) -- (5.5+4,-1.5);

\node at (7+4,-1.25) {Identity matrix;};

\draw[fill = white] (9+4,-1.5) rectangle (9.5+4,-1);
\node at (10.75+4,-1.25) {Zero matrix;};

\end{tikzpicture}
\caption{Pictorial factorization of the $W$ factor in Equation~\eqref{equation_main_factorization} when $\kappa = 3$.}
\end{center}
\label{figure_W_factor}
\end{figure}
\FloatBarrier

\subsection{The algorithm}
We now, in detail, describe an algorithm for computing a symmetric
factorization of a SPD-HODLR matrix.

\begin{itemize}
\item[\textsc{Step 1}:] The first step in the algorithm is to factor
  out the block diagonal matrix on the leaf level, i.e., obtain any
  symmetric factorization of the form:
  \begin{align}
    \begin{bmatrix}
      A_{11}^{(\kappa)} & 0 & 0 & \cdots & 0\\
      0 & A_{22}^{(\kappa)} & 0 & \cdots & 0\\
      0 & 0 & A_{33}^{(\kappa)} & \cdots & 0\\
      \vdots & \vdots & \vdots & \ddots & 0\\
      0 & 0 & 0 & \cdots & A_{2^{\kappa},2^{\kappa}}^{(\kappa)}\\
    \end{bmatrix}
    = A_{\kappa} A_{\kappa}^T
    \label{equation_symmetric_factorization}
  \end{align}
  where
\begin{align}
  A_{\kappa} = \begin{bmatrix}
    W_{11}^{(\kappa)} & 0 & 0 & \cdots & 0\\
    0 & W_{22}^{(\kappa)} & 0 & \cdots & 0\\
    0 & 0 & W_{33}^{(\kappa)} & \cdots & 0\\
    \vdots & \vdots & \vdots & \ddots & 0\\
0 & 0 & 0 & \cdots & W_{2^{\kappa},2^{\kappa}}^{(\kappa)}\\
\end{bmatrix}
\end{align}
and $A_{ii}^{(\kappa)} = W_{ii}^{(\kappa)} W_{ii}^{(\kappa)^T}$. The
computational cost of this step is $\mathcal{O}(n)$.

\begin{remark}
Throughout this algorithm, it is worth recalling the fact that any diagonal sub-block of a SPD matrix is also SPD.
\end{remark}

\item[\textsc{Step 2}:] The next step is to factor out $A_{\kappa}$
  and $A_{\kappa}^T$ from the left and right respectively, i.e., write
$$A = A_{\kappa}\tilde{A}_{\kappa} A_{\kappa}^T.$$ Note that when
  factoring out $A_{\kappa}$ we only need to apply the inverse of
  $W_{ii}^{(\kappa)}$ to each of the $U_i^{(k)}$ at all levels
  $k$. Hence, we need to apply the inverse of $W_{ii}^{(\kappa)}$ to
  $p \kappa$ column vectors. Since the size of $W_{ii}^{(\kappa)}$ is
  $\mathcal{O}(p)$ and there are $2^{\kappa}$ such matrices, the cost
  of this step is $\mathcal{O}\left(\kappa n\right)$.

\begin{remark}
Since the matrix is symmetric, no additional work needs to be done for
factoring out $A_{\kappa}^T$ on the right.
\end{remark}

\item[\textsc{Step 3}:]
Now note that the matrix $\tilde{A}_{\kappa}$ can be written as
\begin{align}
\tilde{A}_{\kappa} =
\begin{bmatrix}
A_{11}^{(\kappa-1,1)}
& U_1^{(\kappa-1,1)} K_{1,2}^{(\kappa-1)} U_2^{(\kappa-1,1)^T}
& \cdots & \cdots\\
U_2^{(\kappa-1,1)} K_{2,1}^{(\kappa-1)} U_1^{(\kappa-1,1)^T}
&
A_{22}^{(\kappa-1,1)}
& \cdots & \cdots\\
\vdots & \vdots & \ddots & \cdots\\
\vdots & \vdots & \vdots &
A_{2^{\kappa-1},2^{\kappa-1}}^{(\kappa-1,1)}
\end{bmatrix}
\end{align}
where
\begin{align}
A_{ii}^{(\kappa-1,1)} = \begin{bmatrix}
I & U_{2i-1}^{(\kappa,1)} K_{2i-1, 2i}^{(\kappa)} U_{2i}^{(\kappa,1)^T}\\
U_{2i}^{(\kappa,1)} K_{2i,2i-1}^{(\kappa)} U_{2i-1}^{(\kappa,1)^T} & I
\end{bmatrix}
\end{align}
and $U_j^{(k,1)}$ indicates that $U_j^{(k)}$ has been updated when
factoring out $A_{\kappa}$. Using Theorem~\ref{thm_main}, we can now
obtain the symmetric factorization of the diagonal blocks:
\begin{align}
\begin{bmatrix}
A_{11}^{(\kappa-1,1)} & 0 & 0 & \cdots & 0\\
0 & A_{22}^{(\kappa-1,1)} & 0 & \cdots & 0\\
0 & 0 & A_{33}^{(\kappa-1,1)} & \cdots & 0\\
\vdots & \vdots & \vdots & \ddots & 0\\
0 & 0 & 0 & \cdots & A_{2^{\kappa-1},2^{\kappa-1}}^{(\kappa-1,1)}\\
\end{bmatrix}
= A_{\kappa-1} A_{\kappa-1}^T
\label{equation_symmetric_factorization_1}
\end{align}
where
\begin{align}
A_{\kappa-1} = 
\begin{bmatrix}
W_{11}^{(\kappa-1)} & 0 & 0 & \cdots & 0\\
0 & W_{22}^{(\kappa-1)} & 0 & \cdots & 0\\
0 & 0 & W_{33}^{(\kappa-1)} & \cdots & 0\\
\vdots & \vdots & \vdots & \ddots & 0\\
0 & 0 & 0 & \cdots & W_{2^{\kappa-1},2^{\kappa-1}}^{(\kappa-1)}\\
\end{bmatrix}
\end{align}
 and 
\begin{align}
W_{ii}^{(\kappa-1)} W_{ii}^{(\kappa-1)^T} = A_{ii}^{(\kappa-1,1)}
\end{align}
with
\begin{align}
A_{ii}^{(\kappa-1,1)} = I + \begin{bmatrix}U_{2i-1}^{(\kappa-1,1)} & 0\\ 0 & U_{2i}^{(\kappa-1,1)}\end{bmatrix} \begin{bmatrix}0 & K_{2i-1,2i}^{(\kappa-1)}\\ K_{2i,2i-1}^{(\kappa-1)} & 0\end{bmatrix} \begin{bmatrix}U_{2i-1}^{(\kappa-1,1)} & 0\\ 0 & U_{2i}^{(\kappa-1,1)}\end{bmatrix}^T.
\end{align}
Hence, the cost of obtaining $A_{\kappa-1}$ is $\mathcal{O}(n)$.

\item[\textsc{Step 4}:] As before, the next step is to factor out
  $A_{\kappa-1}$ and $A_{\kappa-1}^T$ from the left and right sides,
  respectively, i.e., write 
\begin{align}
\tilde{A}_{\kappa} = A_{\kappa-1}
  \tilde{A}_{\kappa-1} A_{\kappa-1}^T
\end{align}
Indeed, we now have:
\begin{align}
A = A_{\kappa} A_{\kappa-1} \tilde{A}_{\kappa-1} A_{\kappa-1}^T A_{\kappa}^T
\end{align}
From Theorem~\ref{thm_main}, since $W_{ii}^{(\kappa-1)}$ is a low-rank
perturbation to identity, the computational cost of applying the
inverse of $W_{ii}^{(\kappa-1)}$ (using the Sherman-Morrison-Woodbury
formula) to all of the $U_i^{(k,1)}$, where $k$ ranges from $1$ to
$\kappa-1$, is $\mathcal{O}((\kappa-1) n/2^{\kappa-1} )$. Since
there are $2^{\kappa-1}$ such matrices $W_{ii}^{(\kappa-1)}$, the net cost of
factoring out $A_{\kappa-1}$ is $\mathcal{O}((\kappa-1)n)$.

\item[\textsc{Step 5}:] Steps 2 through 4 are repeated until we reach
  level $0$, yielding a total computational cost of
\begin{align}
\sum_{k=0}^\kappa
  \mathcal{O}((k+1)n) = \mathcal{O}(\kappa^2 n) =
    \mathcal{O}(n \log^2 n),
\end{align}
where we have used the fact that $\kappa = \mathcal{O}(\log n)$.
\end{itemize}

\begin{remark}
If a nested low-rank structure exists for the off-diagonal blocks,
i.e., if the HODLR matrix can in fact be represented as a
Hierarchically Semi-Separable matrix, then the computational cost of
obtaining the factorization is $\mathcal{O}(n)$.
\end{remark}

\section{Numerical benchmarks}
\label{section_numerical}
We first present some numerical benchmarks for the symmetric
factorization of a low-rank update to a banded matrix, and then demonstrate
the performance of our symmetric factorization scheme on HODLR matrices. These numerical benchmarks can be reproduced by the implementation, which has been made available public at \url{https://gitlab.com/SAFRAN/HODLR}. All the numerical results illustrated in this article were carried out using a single core of a 2.8 GHz Intel Core i7 processor. The implementation made available at \url{https://gitlab.com/SAFRAN/HODLR} has also been parallelized using OpenMP, though to illustrate the scaling, we have not enabled the parallelization.

\subsection{Fast symmetric factorization of low-rank updates to banded matrices}
\label{sec-fast}
To highlight the computational speedup gained by this new factorization,
we compare the time taken for our algorithm with the time taken for 
Cholesky-based factorizations for low-rank updates~\cite{stewart1998matrix}.

\subsubsection{Example 1}
\label{example_1}
Consider a set of $n$ points randomly distributed in the interval
$[-1,1]$. Let the $i,j$ entry of the matrix $A$ be given as $A(i,j) =
\sigma^2 \delta_{ij} + x_i x_j$, where $x_i$ denotes the location of
the $i^{th}$ point. Note that the matrix $A$ is a rank $1$ update to a
scaled identity matrix and can be written as $A = \sigma^2 I + \bx
\bx^T$, where $\bx = \begin{bmatrix}x_1 & \cdots &
  x_n \end{bmatrix}^T$. It is easy to check that the matrix $A$ is
SPD. In fact, in the context of Gaussian processes, this is the
covariance matrix that arises when using a linear covariance function
given by $K(x_i,x_j) = x_i x_j$. Taking $\sigma=1$, we have $A = I +
\bx \bx^T$.
\begin{figure}[!tbp]
\begin{center}
\begin{tikzpicture}[scale=0.95]
\begin{loglogaxis}[
xmin	=	1000,
xmax=	1000000000,
ymin =	0.01,
ymax=	1,
xlabel=	System size,
ylabel=	Time taken in seconds,
legend style={
at={(0.325,0.8)},
anchor=south west}
]
\addplot coordinates {
(8000000, 0.01)
(16000000, 0.02)
(32000000, 0.04)
(64000000, 0.09)
(128000000, 0.18)
(256000000, 0.37)
};
\addplot coordinates {
(2000, 0.02)
(4000, 0.08)
(8000, 0.36)
};
\addplot coordinates {
(1000000, 0.01)
(2000000, 0.02)
(4000000, 0.04)
(8000000, 0.08)
(16000000, 0.16)
(32000000, 0.32)
};
\addplot coordinates {
(10000, 0.01)
(20000, 0.04)
(40000, 0.16)
(80000, 0.64)
};
\legend{Fast symmetric factorization, Cholesky update, $\mathcal{O}(n)$ scale, $\mathcal{O}(n^2)$ scale}
%\addlegendentry{New algorithm}
\end{loglogaxis}
\end{tikzpicture}
\caption{Time taken (in seconds) versus system size for the
  fast algorithm of this paper and the Cholesky update.}
\label{figure_1}
\end{center}
\end{figure}
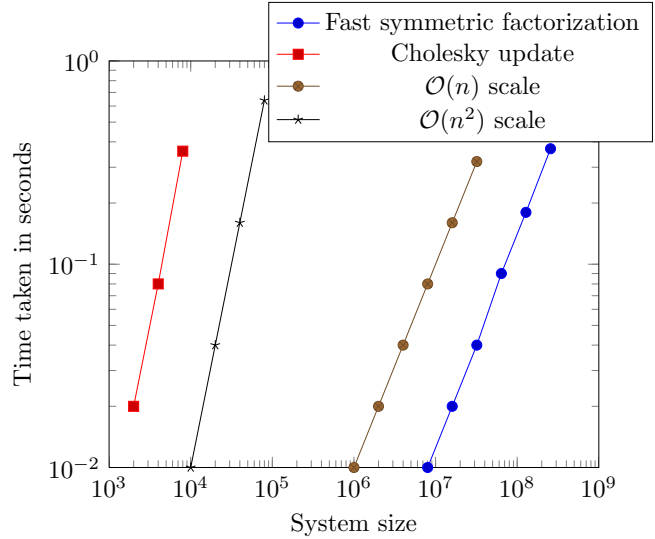

%%%%\FloatBarrier

From Corollary~\ref{cor_1}, the only quantity we need to compute is
$\alpha = \left( \sqrt{1+\Vert \bx \Vert_2^2}-1 \right) / \Vert \bx
\Vert_2^2$. Figure~\ref{figure_1} compares the time taken to obtain
the fast symmetric factorization versus the Cholesky
factorization of $A$.

%%%%\FloatBarrier

\subsubsection{Example 2}
%%%%\begin{ex}
\label{example_2}
Next, we consider a rank $r$ update to a SPD tridiagonal matrix:
\begin{align}
\label{equation_tridiagonal_update}
A = T + UU^T
\end{align}
where $U \in \mathbb{R}^{n \times r}$ and $T \in \mathbb{R}^{n \times
  n}$ is a tridiagonal matrix with entries given as:
\begin{align}
\label{equation_Tridiagonal}
T(i,j) =
\begin{cases}
1 & \text{if }i=j=1\\
2 & \text{if }i=j>1 \\
-1 & \text{if }\vert i- j \vert = 1 \\
0 & \text{otherwise}.
\end{cases}
\end{align}
Note that the Cholesky factorization of $T = LL^T$ is 
\begin{align}
\label{equation_Cholesky_Factor}
L(i,j) =
\begin{cases}
1 & \text{if }i=j\\
-1 & \text{if }i-j=1\\
0 & \text{otherwise},
\end{cases}
\end{align}
where $L \in \mathbb{R}^{n \times n}$. The first step to obtain the
symmetric factorization of $LL^T+UU^T$ is to first factor out $L$ to
obtain
\begin{equation}
LL^T + UU^T = L(I+\tilde{U}\tilde{U}^T) L^T,
\end{equation}
where $\tilde U$ is such that $L\tilde{U} = U$. The next step is to obtain the
symmetric factorization of $I+\tilde{U}\tilde{U}^T$ as
$(I+\tilde{U}X\tilde{U}^T)(I+\tilde{U}X^T\tilde{U}^T)$. Now the
symmetric factorization of $T+UU^T$ is given as
\begin{align}
T+UU^T = 
%\underbrace{L
L(I+\tilde{U}X\tilde{U}^T)
%}_{L_1}
(I+\tilde{U}X\tilde{U}^T)^TL^T.
\end{align}

Figure~\ref{figure_compare} compares the time taken versus system size
for different low-rank perturbation, while Figure~\ref{figure_5} fixes
the system size and plots the time taken to obtain the symmetric
factorization versus the rank of the low-rank perturbation.

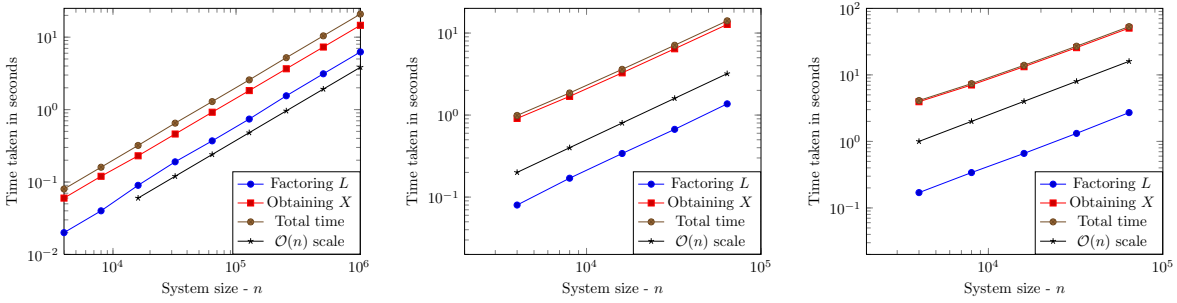
\begin{figure}[!tbp]
\begin{center}
%%%First subfigure
\subfigure{
\begin{tikzpicture}[scale=0.575]
\begin{loglogaxis}[
xmin	=	4000,
xmax=	1025000,
ymin =	0.01,
ymax=	25,
xlabel=	System size - $n$,
ylabel=	Time taken in seconds,
legend style={
at={(1,0)},
anchor=south east}
]
\addplot coordinates {
(4000, 0.02)
(8000, 0.04)
(16000, 0.09)
(32000, 0.19)
(64000, 0.37)
(128000, 0.74)
(256000, 1.55)
(512000, 3.12)
(1024000, 6.25)
};
\addplot coordinates {
(4000, 0.06)
(8000, 0.12)
(16000, 0.23)
(32000, 0.46)
(64000, 0.92)
(128000, 1.83)
(256000, 3.66)
(512000, 7.31)
(1024000, 14.59)
};
\addplot coordinates {
(4000, 0.08)
(8000, 0.16)
(16000, 0.32)
(32000, 0.65)
(64000, 1.29)
(128000, 2.57)
(256000, 5.21)
(512000, 10.43)
(1024000, 20.84)
};
\addplot coordinates {
(16000, 0.06)
(32000, 0.12)
(64000, 0.24)
(128000, 0.48)
(256000, 0.96)
(512000, 1.92)
(1024000, 3.84)
};
\legend{Factoring $L$, Obtaining $X$, Total time, $\mathcal{O}(n)$ scale}
%\addlegendentry{New algorithm}
\end{loglogaxis}
\end{tikzpicture}
\label{figure_2}
}
%%%Second subfigure
\subfigure{
\begin{tikzpicture}[scale=0.575]
\begin{loglogaxis}[
xmin	=	2000,
xmax=	100000,
ymin =	0.02,
ymax=	20,
xlabel=	System size - $n$,
ylabel=	Time taken in seconds,
legend style={
at={(1,0)},
anchor=south east}
]
\addplot coordinates {
(4000, 0.08)
(8000, 0.17)
(16000, 0.34)
(32000, 0.67)
(64000, 1.37)
};
\addplot coordinates {
(4000, 0.91)
(8000, 1.69)
(16000, 3.27)
(32000, 6.43)
(64000, 12.73)
};
\addplot coordinates {
(4000, 0.99)
(8000, 1.86)
(16000, 3.61)
(32000, 7.1)
(64000, 14.1)
};
\addplot coordinates {
(4000, 0.2)
(8000, 0.4)
(16000, 0.8)
(32000, 1.6)
(64000, 3.2)
};
\legend{Factoring $L$, Obtaining $X$, Total time, $\mathcal{O}(n)$ scale}
%\addlegendentry{New algorithm}
\end{loglogaxis}
\end{tikzpicture}
\label{figure_3}
}
%%%Third subfigure
\subfigure{
\begin{tikzpicture}[scale=0.575]
\begin{loglogaxis}[
xmin	=	2000,
xmax=	100000,
ymin =	0.02,
ymax=	100,
xlabel=	System size - $n$,
ylabel=	Time taken in seconds,
legend style={
at={(1,0)},
anchor=south east}
]
\addplot coordinates {
(4000, 0.17)
(8000, 0.34)
(16000, 0.66)
(32000, 1.32)
(64000, 2.71)
};
\addplot coordinates {
(4000, 3.95)
(8000, 7.05)
(16000, 13.27)
(32000, 25.67)
(64000, 50.48)
};
\addplot coordinates {
(4000, 4.12)
(8000, 7.39)
(16000, 13.93)
(32000, 26.99)
(64000, 53.19)
};
\addplot coordinates {
(4000, 1)
(8000, 2)
(16000, 4)
(32000, 8)
(64000, 16)
};
\legend{Factoring $L$, Obtaining $X$, Total time, $\mathcal{O}(n)$ scale}
%\addlegendentry{New algorithm}
\end{loglogaxis}
\end{tikzpicture}
\label{figure_4}
}
\caption{Time taken (in secs) by the algorithm for Example~\ref{example_2} with different ranks $r$. Left: $r=25$, Middle: $r=100$, Right: $r=200$.}
\label{figure_compare}
\end{center}
\end{figure}
\FloatBarrier

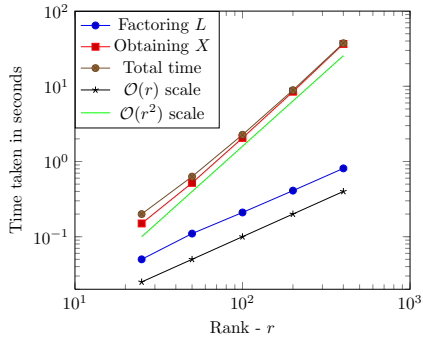
\begin{figure}[!htbp]
\begin{center}
\begin{tikzpicture}[scale=0.65]
\begin{loglogaxis}[
xmin	=	10,
xmax=	1000,
ymin =	0.02,
ymax=	100,
xlabel=	Rank - $r$,
ylabel=	Time taken in seconds,
legend style={
at={(0,1)},
anchor=north west}
]
\addplot coordinates {
(25, 0.05)
(50, 0.11)
(100, 0.21)
(200, 0.41)
(400, 0.81)
};
\addplot coordinates {
(25, 0.15)
(50, 0.52)
(100, 2.06)
(200, 8.48)
(400, 36.61)
};
\addplot coordinates {
(25, 0.2)
(50, 0.63)
(100, 2.27)
(200, 8.89)
(400, 37.42)
};
\addplot coordinates {
(25, 0.025)
(50, 0.05)
(100, 0.1)
(200, 0.2)
(400, 0.4)
};
\addplot[green] coordinates {
(25, 0.1)
(50, 0.4)
(100, 1.6)
(200, 6.4)
(400, 25.6)
};
\legend{Factoring $L$, Obtaining $X$, Total time, $\mathcal{O}(r)$ scale, $\mathcal{O}(r^2)$ scale}
%\addlegendentry{New algorithm}
\end{loglogaxis}
\end{tikzpicture}
\caption{Time taken (in secs) by the algorithm for Example~\ref{example_2} with $n=10000$.}
\label{figure_5}
\end{center}
\end{figure}
%%%%\end{ex}
%%%%\FloatBarrier

\subsection{Fast symmetric factorization of HODLR matrices}

We now present numerical benchmarks for applying our symmetric
factorization algorithm to Hierarchical Off-Diagonal Low-Rank Matrices.
In particular, we show results for symmetrically factorizing covariance matrices
and the {\em mobility matrix} encountered in hydrodynamic fluctuations.

The first subsection (Section~\ref{sec-gauss}) contains results for applying the algorithm
to a covariance matrix whose entires are obtained by
evaluating a Gaussian covariance kernel acting on data-points in 
three dimensions. The following subsection (Section~\ref{sec-biharmonic} )
contains analogous results for the covariance matrix obtained
from a biharmonic covariance kernel evaluated on data-points in one and two 
dimensions. In fact, this matrix also occurs when performing
second-order radial basis function interpolation.
In the next subsection (Section~\ref{sec-rotne}), we apply the symmetric factorization to the Rotne-Prager-Yamakawa (RPY) tensor. This tensor
serves as a model for the hydrodynamic forces between spheres of
constant radii. We distribute sphere locations along one-, two-, 
and three-dimensional manifolds all embedded in three dimensions. Finally, in the last subsection (Section~\ref{sec-matern}), we present numerical results for the Mat\'{e}rn covariance kernel.

In all numerical examples, the points in all dimensions are ordered
based on a kd-tree. The low-rank decomposition of the
off-diagonal blocks is obtained using a slight modification of the
adaptive cross
approximation~\cite{rjasanow2002adaptive,zhao2005adaptive} algorithm,
which is essentially a variant of the partially pivoted LU algorithm.

\subsubsection{A Gaussian covariance kernel}
\label{sec-gauss}
Covariance matrices constructed using positive-definite parametric
covariance kernels arise frequently when performing nonparametric
regression using Gaussian processes. Related results for similar
algorithms can be found in~\cite{ambikasaran2014fast}. The entries of
covariance matrices corresponding to a Gaussian covariance kernel are
calculated as
\begin{align}
K(i,j) = \sigma^2 \delta_{ij} + \exp \left(-\Vert r_i
- r_j \Vert^2\right),
\end{align}
where $\sigma^2$ is proportional to the inherent measurement noise in
the underlying regression model. In our numerical experiments, we set
$\sigma = 1$ and distribute the points $r_i$ randomly in the cube
$[-1, 1]^3$.  The timings matrix are presented in Table~\ref{table_Gaussian} and the scaling is presented in Figure~\ref{figure_Gaussian}.

% Gaussian
\begin{table}[!htbp]
	\begin{center}
	\caption{Total time in seconds versus system size in different dimensions; The input tolerance is $10^{-12}$ for 1D; $10^{-9}$ for 2D; $10^{-6}$ for 3D;}
	\begin{tabular}{|c|c|c|c|}
		\hline
		\multirow{2}{*}{System size} & \multicolumn{3}{|c|}{Time taken in seconds}\\
		\cline{2-4}
		& 1D & 2D & 3D\\
		\hline
		\hline
		\rowcolor{gray!30}
		$64 \times 10^3$ & $0.74$ & $1.88$ & $4.82$\\
		\hline
		$128 \times 10^3$ & $1.51$ & $3.91$ & $9.84$\\
		\hline
		\rowcolor{gray!30}
		$256 \times 10^3$ & $3.26$ & $8.08$ & $20.51$\\
		\hline
		$512 \times 10^3$ & $6.84$ & $17.37$ & $46.78$\\
		\hline
		\rowcolor{gray!30}
		$1024 \times 10^3$ & $13.85$ & $40.26$ & $101.65$\\
		\hline
	\end{tabular}
	\label{table_Gaussian}
	\end{center}
\end{table}

% Gaussian
\begin{figure}[!htbp]
	\subfigure[1D] {
		\includegraphics[width=0.33\textwidth]{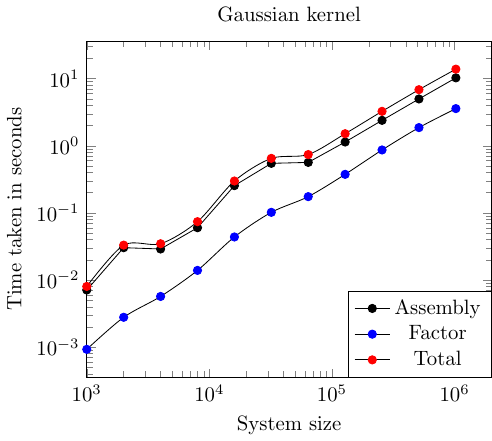}
	}
	\subfigure[2D] {
		\includegraphics[width=0.33\textwidth]{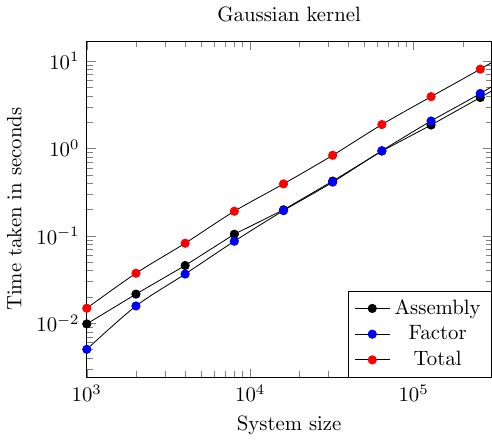}
	}
	\subfigure[3D] {
		\includegraphics[width=0.33\textwidth]{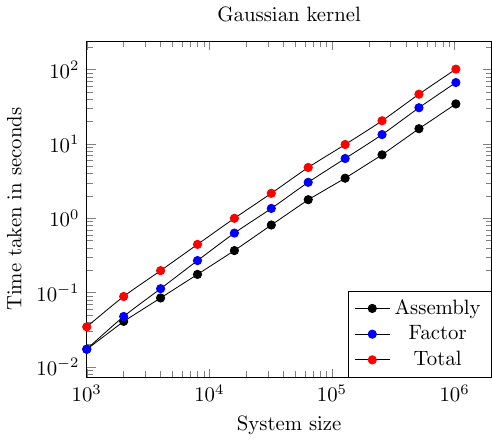}
	}
	\caption{Time taken versus system size for Gaussian kernel in different dimensions; The input tolerance was $10^{-12}$ for 1D; $10^{-9}$ for 2D; $10^{-6}$ for 3D;}
	\label{figure_Gaussian}
\end{figure}

\FloatBarrier

\subsubsection{A biharmonic covariance kernel}
\label{sec-biharmonic}
As in the previous section, a covariance matrices arising in Gaussian
processes can be modeled using the biharmonic covariance kernel, where
the $i,j$ entry is given as
\begin{align}
K(i,j) = \sigma^2 \delta_{ij} + \dfrac{r_{ij}^2}{a^2} \log\left(\dfrac{r_{ij}}a \right),
\end{align}
where the parameters $\sigma$ and $a$ are chosen so that the matrix is
positive-definite. Figure~\ref{figure_biharmonic} presents the results for the Biharmonic kernel in different dimensions.

% Biharmonic
\begin{figure}[!htbp]
	\subfigure[1D] {
		\includegraphics[width=0.33\textwidth]{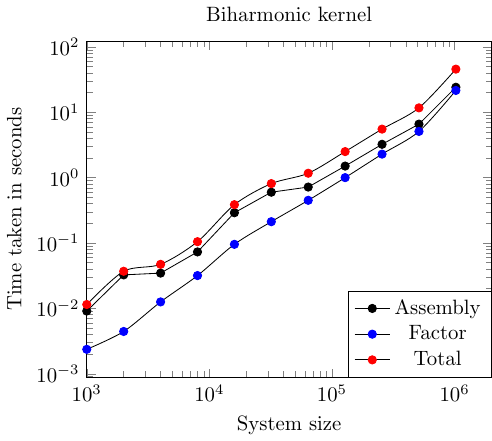}
	}
	\subfigure[2D] {
		\includegraphics[width=0.33\textwidth]{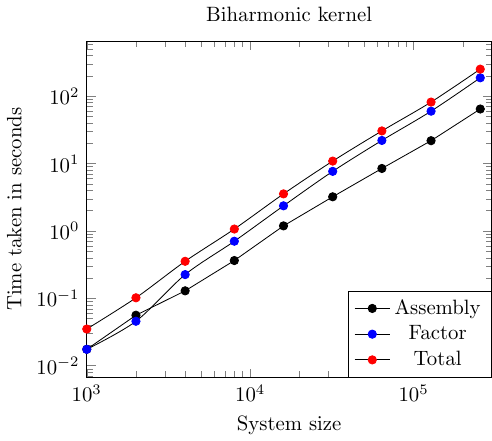}
	}
	\subfigure[3D] {
		\includegraphics[width=0.33\textwidth]{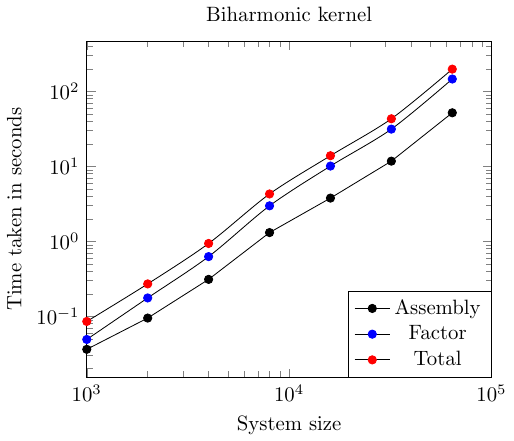}
	}
	\caption{Time taken versus system size for biharmonic kernel in different dimensions; The input tolerance was $10^{-12}$ for 1D; $10^{-9}$ for 2D; $10^{-6}$ for 3D;}
	\label{figure_biharmonic}
\end{figure}

\FloatBarrier

\subsubsection{The Rotne-Prager-Yamakawa tensor}
\label{sec-rotne}
The RPY diffusion tensor is frequently used to model the hydrodynamic
interactions in simulations of Brownian dynamics. The RPY tensor is
defined as
\begin{align}
D_{ij} = \begin{cases} \dfrac{k_BT}{6 \pi \eta a} \left[
    \left(1-\dfrac9{32}\dfrac{r_{ij}}a\right)\mathbf{I} +
    \dfrac3{32a}\dfrac{\mathbf{r}_{ij} \otimes
      \mathbf{r}_{ij}}{r_{ij}}\right] & r<2a,\\ 
\dfrac{k_BT}{8 \pi \eta
    r_{ij}} \left[\mathbf{I} + \dfrac{\mathbf{r}_{ij}
      \otimes \mathbf{r}_{ij}}{r_{ij}^2} +\dfrac{2a^2}{3r_{ij}^2}
    \left( \mathbf{I} - 3\dfrac{\mathbf{r}_{ij} \otimes
      \mathbf{r}_{ij}}{r_{ij}^2} \right)\right] &
  r\geq2a, \end{cases}
\end{align}
where $k_B$ is the Boltzmann constant, $T$ is the absolute
temperature, $\eta$ is the viscosity of the fluid $a$ is the
hydrodynamic radius of the particles, $r_{ij}$ is the distance between
the $i^{th}$ and $j^{th}$ particle and $\mathbf{r}_{ij}$ is the vector
connecting the $i^{th}$ and $j^{th}$ particles. The tensor is obtained as an
approximation to the Stokes flow around two spheres by neglecting the
hydrodynamic rotation-rotation and rotation-translation coupling.  The
resulting matrix is often referred to as the mobility matrix. It has
been shown in~\cite{yamakawa2003transport} that this tensor is
positive-definite for all particle configurations. Fast symmetric
factorizations of the RPY tensor are crucial in Brownian dynamics
simulations. Geyer and Winter~\cite{geyer2009n2} discuss an
$\mathcal{O}(n^2)$ algorithm for approximating the square-root of the
RPY tensor. Jiang et al.~\cite{jiang2013fast} discuss an approximate
algorithm, which relies on a Chebyshev spectral approximation of the
matrix square-root coupled with a FMM. Their method
scales as $\mathcal{O}(\sqrt{\kappa} n)$, where $\kappa$ is the
condition number of the RPY tensor. Our algorithm scales as
$\mathcal{O}(n \log^2 n)$ if the particles are located along a line,
$\mathcal{O}(n^2)$ if the particles are distributed on a surface, and
as $\mathcal{O}(n^{7/3})$, if the particles are distributed in a
three-dimensional volume.

\begin{remark}
Note that since the RPY tensor is singular, on $2$D and $3$D manifolds
the ranks of the off-diagonal blocks would grow as
$\mathcal{O}(n^{1/2})$ and $\mathcal{O}(n^{2/3})$, respectively. Since
the computational cost of the symmetric factorization scales as
$\mathcal{O}(p^2n)$, the computational cost for the symmetric
factorization to scale as $\mathcal{O}(n^2)$ and
$\mathcal{O}(n^{7/3})$ on $2$D and $3$D manifolds, respectively. The
numerical benchmarks, plotted in Figure~\ref{figure_RPY}, also validate this scaling of our algorithm in all
three configurations.
\end{remark}

% RPY
\begin{figure}[!htbp]
	\subfigure[1D] {
		\includegraphics[width=0.33\textwidth]{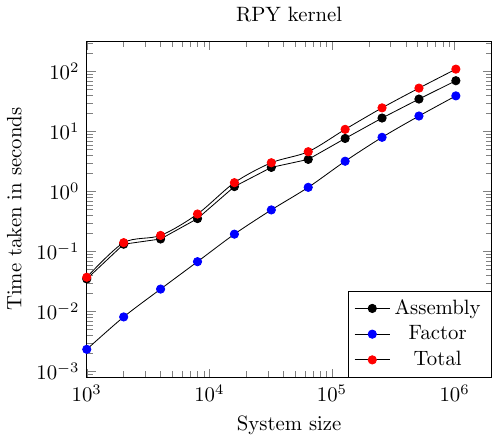}
	}
	\subfigure[2D] {
		\includegraphics[width=0.33\textwidth]{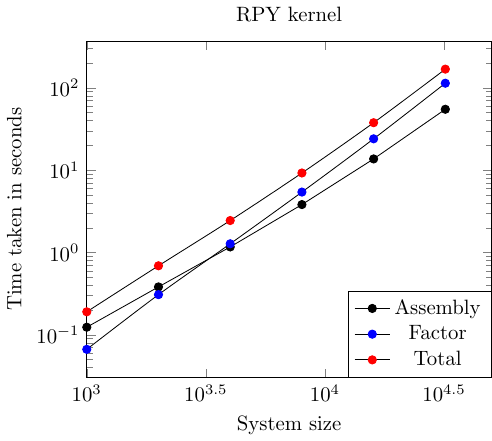}
	}
	\subfigure[3D] {
		\includegraphics[width=0.33\textwidth]{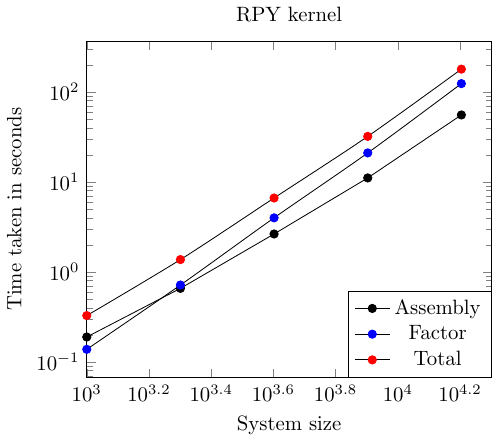}
	}
	\caption{Time taken versus system size for the RPY tensor in different dimensions; The input tolerance was $10^{-12}$ for 1D; $10^{-9}$ for 2D; $10^{-6}$ for 3D;}
	\label{figure_RPY}
\end{figure}

\subsubsection{The Mat\'{e}rn Kernel}
\label{sec-matern}
The Mat\'{e}rn covariance function is frequently used in spatial statistics, geostatistics, Gaussian process regression in machine learning, etc. The covariance function is given by
\begin{align}
	C_{\nu}(r) & = \sigma^2 \dfrac{2^{1-\nu}}{\Gamma(\nu)} \left(\dfrac{r\sqrt{2\nu}}{\rho}\right)^{\nu} K_{\nu}\left(\dfrac{r\sqrt{2\nu}}{\rho}\right)
	\label{eqn_matern}
\end{align}
where $K_{\nu}$ is the modified Bessel function of the second kind, $\rho$, $\nu$ are non-negative parameters. When $\nu$ is a half integer, i.e., $\nu = p+1/2$, then Equation~\eqref{eqn_matern} simplifies to
\begin{align}
	C_{p+1/2}(r) & = \sigma^2 \exp\left(-\dfrac{r\sqrt{2\nu}}{\rho}\right) \dfrac{\Gamma(p+1)}{\Gamma(2p+1)} \displaystyle \sum_{i=0}^p \dfrac{(p+i)!}{i!(p-i)!} \left(\dfrac{r\sqrt{8 \nu}}{\rho}\right)^{p-i}
\end{align}
The expressions for the first few Mat\'{e}rn kernels are displayed in Table~\ref{table_matern_kernels}. The scaling of the time taken as system size for the first two Mat\'{e}rn kernels are plotted in Figure~\ref{figure_Matern_1} and~\ref{figure_Matern_2}.
\begin{table}[!htbp]
	\begin{center}
		\caption{Mat\'{e}rn kernel for first few values of $p$}
	\begin{tabular}{|c|c|}
		\hline
		\hline
		$p=0$ & $\sigma^2 \exp\left(-\dfrac{r}{\rho}\right)$\\
		\hline
		\hline
		$p=1$ & $\sigma^2 \left(1+\dfrac{r\sqrt3}{\rho}\right)\exp\left(-\dfrac{r\sqrt3}{\rho}\right)$\\
		\hline
		\hline
		$p=1$ & $\sigma^2 \left(1+\dfrac{r\sqrt5}{\rho} + \dfrac{5r^2}{3\rho^2}\right)\exp\left(-\dfrac{r\sqrt5}{\rho}\right)$\\
		\hline
		\hline
	\end{tabular}
	\end{center}
	\label{table_matern_kernels}
\end{table}

% Matern 1
\begin{figure}[!htbp]
	\subfigure[1D] {
		\includegraphics[width=0.33\textwidth]{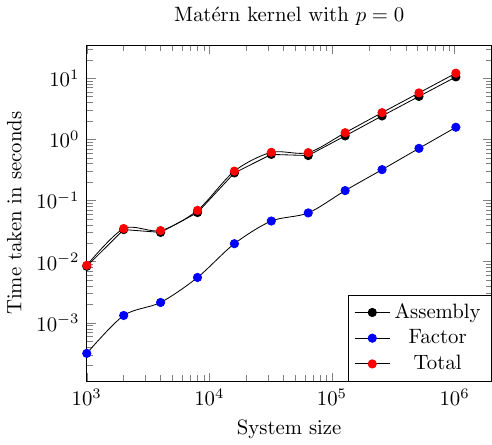}
	}
	\subfigure[2D] {
		\includegraphics[width=0.33\textwidth]{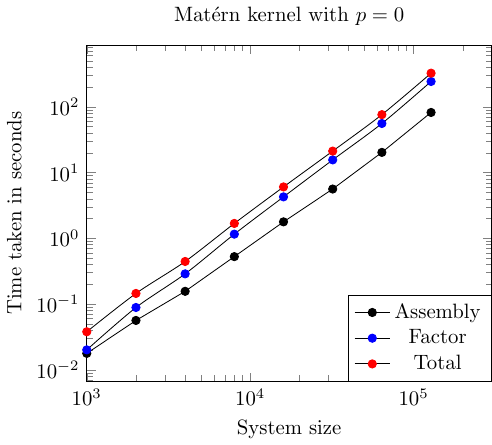}
	}
	\subfigure[3D] {
		\includegraphics[width=0.33\textwidth]{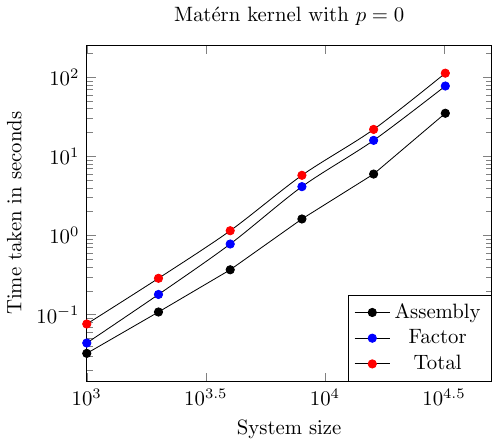}
	}
	\caption{Time taken versus system size for the Mat\'{e}rn kernel with $p=0$ in different dimensions; The input tolerance was $10^{-12}$ for 1D; $10^{-9}$ for 2D; $10^{-6}$ for 3D;}
	\label{figure_Matern_1}
\end{figure}

% Matern 2
\begin{figure}[!htbp]
	\subfigure[1D] {
		\includegraphics[width=0.33\textwidth]{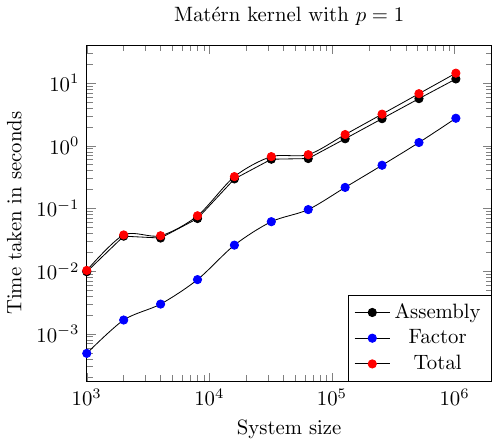}
	}
	\subfigure[2D] {
		\includegraphics[width=0.33\textwidth]{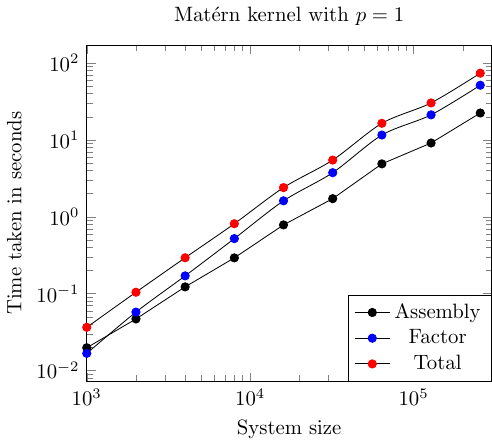}
	}
	\subfigure[3D] {
		\includegraphics[width=0.33\textwidth]{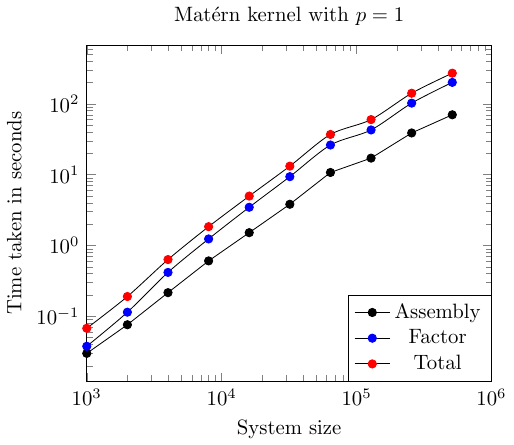}
	}
	\caption{Time taken versus system size for the Mat\'{e}rn kernel with $p=1$ in different dimensions; The input tolerance was $10^{-12}$ for 1D; $10^{-9}$ for 2D; $10^{-6}$ for 3D;}
	\label{figure_Matern_2}
\end{figure}

\FloatBarrier

\section{Conclusion}
\label{section_conclusion}
The article discusses a fast symmetric factorization for a class of
symmetric positive-definite hierarchically structured matrices. Our 
symmetric factorization algorithm is based on two ingredients: a 
novel formula for the symmetric factorization of a low-rank update to
the identity, and a recursive divide-and-conquer strategy compatible
with hierarchically structured matrices.

In the case where the hierarchical structure present is that of
Hierarchically Off-Diagonal Low-Rank matrices, the algorithm scales as
$\mathcal{O}(n \log^2n)$. The numerical benchmarks for dense
covariance matrix examples validate the scaling.  Furthermore, we also
applied the algorithm to the mobility matrix encountered in
Brownian-hydrodynamics, elements of which are computed from the
Rotne-Prager-Yamakawa tensor.  In this case, since the ranks of
off-diagonal blocks scale as $\mathcal{O}(n^{2/3})$, when the particles
are on a three-dimensional manifold, the algorithm scales as
$\mathcal{O}(n^{7/3})$.  Obtaining an $\mathcal{O}(n)$ symmetric
factorization for the mobility matrix is a subject of ongoing research
within our group.

It is also worth noting that with nested low-rank
basis of the off-diagonal blocks, i.e., if the HODLR matrices are
assumed to have a Hierarchical Semi-Separable structure (HSS) instead, then
the computational cost of the algorithm would scale as
$\mathcal{O}(p^2n)$. Extensions to this case is relatively straightforward.

%%%%%\section{Appendix}

\bibliographystyle{plain}
%\bibliography{./biblio_Siva.bib}

\begin{thebibliography}{10}

\bibitem{akritas1996various}
Alkiviadis~G Akritas, Evgenia~K Akritas, and Genadii~I Malaschonok.
\newblock Various proofs of {S}ylvester's (determinant) identity.
\newblock {\em Mathematics and Computers in Simulation}, 42(4):585--593, 1996.

\bibitem{ambikasaran2013thesis}
Sivaram Ambikasaran.
\newblock {\em {Fast Algorithms for Dense Numerical Linear Algebra}}.
\newblock PhD thesis, Stanford University, August 2013.

\bibitem{ambikasaran2013fast}
Sivaram Ambikasaran and Eric Darve.
\newblock An $\mathcal{O}(n \log n)$ fast direct solver for partial
  hierarchically semi-separable matrices.
\newblock {\em Journal of Scientific Computing}, pages 1--25, 2013.

\bibitem{ambikasaran2014ifmm}
Sivaram Ambikasaran and Eric Darve.
\newblock The inverse fast multipole method.
\newblock {\em arXiv preprint arXiv:1407.1572}, 2014.

\bibitem{ambikasaran2014fast}
Sivaram Ambikasaran, Daniel Foreman-Mackey, Leslie Greengard, David~W Hogg, and
  Michael O'Neil.
\newblock Fast direct methods for {G}aussian processes and the analysis of
  {NASA} \emph{Kepler} mission data.
\newblock {\em arXiv preprint arXiv:1403.6015}, 2014.

\bibitem{ambikasaran2013large}
Sivaram Ambikasaran, Judith~Yue Li, Peter~K Kitanidis, and Eric Darve.
\newblock Large-scale stochastic linear inversion using hierarchical matrices.
\newblock {\em Computational Geosciences}, 17(6):913--927, 2013.

\bibitem{ambikasaran2013fastBayes}
Sivaram Ambikasaran, Arvind~Krishna Saibaba, Eric~F Darve, and Peter~K
  Kitanidis.
\newblock {F}ast algorithms for {B}ayesian inversion.
\newblock In {\em Computational Challenges in the Geosciences}, pages 101--142.
  Springer, 2013.

\bibitem{aminfar2014fast}
Amirhossein Aminfar, Sivaram Ambikasaran, and Eric Darve.
\newblock A fast block low-rank dense solver with applications to
  finite-element matrices.
\newblock {\em arXiv preprint arXiv:1403.5337}, 2014.

\bibitem{bjorck1983schur}
{\AA}ke Bj{\"o}rck and Sven Hammarling.
\newblock A {S}chur method for the square root of a matrix.
\newblock {\em Linear algebra and its applications}, 52:127--140, 1983.

\bibitem{borm2003hierarchical}
Steffen B{\"o}rm, Lars Grasedyck, and Wolfgang Hackbusch.
\newblock Hierarchical matrices.
\newblock {\em Lecture notes}, 21, 2003.

\bibitem{chandrasekaran2006fast1}
Shivkumar Chandrasekaran, Patrick Dewilde, Ming Gu, William Lyons, and Timothy
  Pals.
\newblock A fast solver for {HSS} representations via sparse matrices.
\newblock {\em SIAM Journal on Matrix Analysis and Applications}, 29(1):67--81,
  2006.

\bibitem{chandrasekaran2006fast}
Shivkumar Chandrasekaran, Ming Gu, and Timothy Pals.
\newblock A fast {ULV} decomposition solver for hierarchically semiseparable
  representations.
\newblock {\em SIAM Journal on Matrix Analysis and Applications},
  28(3):603--622, 2006.

\bibitem{chendata}
{J}ie {C}hen.
\newblock Data structure and algorithms for recursively low-rank compressed
  matrices.
\newblock {\em Argonne National Laboratory}, 2014.

\bibitem{cichocki2009nonnegative}
Andrzej Cichocki, Rafal Zdunek, Anh~Huy Phan, and Shunichi Amari.
\newblock {\em Nonnegative matrix and tensor factorizations: applications to
  exploratory multi-way data analysis and blind source separation}.
\newblock John Wiley \& Sons, 2009.

\bibitem{geist2010statistically}
Matthieu Geist and Olivier Pietquin.
\newblock Statistically linearized recursive least squares.
\newblock In {\em Machine Learning for Signal Processing (MLSP), 2010 IEEE
  International Workshop on}, pages 272--276. IEEE, 2010.

\bibitem{geyer2009n2}
Tiham{\'e}r Geyer and Uwe Winter.
\newblock An $\mathcal{O}(n^2)$ approximation for hydrodynamic interactions in
  brownian dynamics simulations.
\newblock {\em The Journal of chemical physics}, 130(11):114905, 2009.

\bibitem{gill1974methods}
Philip~E Gill, Gene~H Golub, Walter Murray, and Michael~A Saunders.
\newblock Methods for modifying matrix factorizations.
\newblock {\em Mathematics of Computation}, 28(126):505--535, 1974.

\bibitem{grasedyck2003construction}
Lars Grasedyck and Wolfgang Hackbusch.
\newblock Construction and arithmetics of $\mathcal{H}$-matrices.
\newblock {\em Computing}, 70(4):295--334, 2003.

\bibitem{greengard1987fast}
Leslie Greengard and Vladimir Rokhlin.
\newblock A fast algorithm for particle simulations.
\newblock {\em Journal of {C}omputational {P}hysics}, 73(2):325--348, 1987.

\bibitem{hackbusch1999sparse}
Wolfgang Hackbusch.
\newblock A sparse matrix arithmetic based on $\mathcal{H}$-matrices. {P}art
  {I}: {I}ntroduction to $\mathcal{H}$-matrices.
\newblock {\em Computing}, 62(2):89--108, 1999.

\bibitem{hackbusch2002data}
Wolfgang Hackbusch and Steffen B{\"o}rm.
\newblock Data-sparse approximation by adaptive $\mathcal{H}^2$-matrices.
\newblock {\em Computing}, 69(1):1--35, 2002.

\bibitem{hackbusch2000h2}
Wolfgang Hackbusch, Boris Khoromskij, and Stefan~A Sauter.
\newblock On $\mathcal{H}^2$-matrices.
\newblock In Hans-Joachim Bungartz, RonaldH.W. Hoppe, and Christoph Zenger,
  editors, {\em Lectures on Applied Mathematics}, pages 9--29. Springer Berlin
  Heidelberg, 2000.

\bibitem{hackbusch2000sparse}
Wolfgang Hackbusch and Boris~N Khoromskij.
\newblock A sparse $\mathcal{H}$-matrix arithmetic.
\newblock {\em Computing}, 64(1):21--47, 2000.

\bibitem{hackbusch2004hierarchical}
Wolfgang Hackbusch, Boris~N Khoromskij, and Ronald Kriemann.
\newblock Hierarchical matrices based on a weak admissibility criterion.
\newblock {\em Computing}, 73(3):207--243, 2004.

\bibitem{higham1986newton}
Nicholas~J Higham.
\newblock Newton's method for the matrix square root.
\newblock {\em Mathematics of Computation}, 46(174):537--549, 1986.

\bibitem{jiang2013fast}
Shidong Jiang, Zhi Liang, and Jingfang Huang.
\newblock A fast algorithm for brownian dynamics simulation with hydrodynamic
  interactions.
\newblock {\em Mathematics of Computation}, 82(283):1631--1645, 2013.

\bibitem{lai2014fast}
Jun Lai, Sivaram Ambikasaran, and Leslie~F Greengard.
\newblock A fast direct solver for high frequency scattering from a large
  cavity in two dimensions.
\newblock {\em arXiv preprint arXiv:1404.3451}, 2014.

\bibitem{li2014kalman}
Judith~Yue Li, Sivaram Ambikasaran, Eric~F Darve, and Peter~K Kitanidis.
\newblock A kalman filter powered by h2-matrices for quasi-continuous data
  assimilation problems.
\newblock {\em Water Resources Research}, 2014.

\bibitem{lyons2005fast}
William Lyons.
\newblock {\em Fast algorithms with applications to PDEs}.
\newblock PhD thesis, University of California Santa Barbara, 2005.

\bibitem{mandel2006efficient}
Jan Mandel.
\newblock {\em Efficient implementation of the ensemble {K}alman filter}.
\newblock University of Colorado at Denver and Health Sciences Center, Center
  for Computational Mathematics, 2006.

\bibitem{martinsson2011fast}
Per-Gunnar Martinsson.
\newblock A fast randomized algorithm for computing a hierarchically
  semiseparable representation of a matrix.
\newblock {\em SIAM Journal on Matrix Analysis and Applications},
  32(4):1251--1274, 2011.

\bibitem{matheron1963principles}
Georges Matheron.
\newblock Principles of geostatistics.
\newblock {\em Economic geology}, 58(8):1246--1266, 1963.

\bibitem{pauca2006nonnegative}
V~Paul Pauca, Jon Piper, and Robert~J Plemmons.
\newblock Nonnegative matrix factorization for spectral data analysis.
\newblock {\em Linear algebra and its applications}, 416(1):29--47, 2006.

\bibitem{rjasanow2002adaptive}
Sergej Rjasanow.
\newblock Adaptive cross approximation of dense matrices.
\newblock {\em IABEM 2002, International Association for Boundary Element
  Methods}, 2002.

\bibitem{ambikasaran2012application}
Arvind~Krishna Saibaba, Sivaram Ambikasaran, Judith Yue~Li, Peter~K Kitanidis,
  and Eric~F Darve.
\newblock {A}pplication of hierarchical matrices to linear inverse problems in
  geostatistics.
\newblock {\em {O}il and {G}as {S}cience and {T}echnology-{R}evue de
  {l'IFP-I}nstitut {F}rancais du {P}etrole}, 67(5):857, 2012.

\bibitem{seeger2007low}
Matthias Seeger et~al.
\newblock Low rank updates for the {C}holesky decomposition.
\newblock {\em University of California at Berkeley, Tech. Rep}, 2007.

\bibitem{starr1991numerical}
Harold~P Starr~Jr.
\newblock On the numerical solution of one-dimensional integral and
  differential equations.
\newblock Technical report, DTIC Document, 1991.

\bibitem{stewart1998matrix}
Gilbert~W Stewart.
\newblock {\em Matrix Algorithms: Volume 1, Basic Decompositions}, volume~1.
\newblock Cambridge University Press, 1998.

\bibitem{tao2012topics}
Terence Tao.
\newblock {\em Topics in random matrix theory}, Volume 132.
\newblock American Mathematical Soc., 2012.

\bibitem{wackernagel2003multivariate}
Hans Wackernagel.
\newblock {\em Multivariate geostatistics}.
\newblock Springer, 2003.

\bibitem{wang2008multi}
Dingding Wang, Tao Li, Shenghuo Zhu, and Chris Ding.
\newblock Multi-document summarization via sentence-level semantic analysis and
  symmetric matrix factorization.
\newblock In {\em Proceedings of the 31st annual international ACM SIGIR
  conference on Research and development in information retrieval}, pages
  307--314. ACM, 2008.

\bibitem{xia2009superfast}
Jianlin Xia, Shivkumar Chandrasekaran, Ming Gu, and Xiaoye~S Li.
\newblock Superfast multifrontal method for large structured linear systems of
  equations.
\newblock {\em SIAM Journal on Matrix Analysis and Applications},
  31(3):1382--1411, 2009.

\bibitem{xia2010robust}
Jianlin Xia and Ming Gu.
\newblock Robust approximate {C}holesky factorization of rank-structured
  symmetric positive definite matrices.
\newblock {\em SIAM Journal on Matrix Analysis and Applications},
  31(5):2899--2920, 2010.

\bibitem{yamakawa2003transport}
Hiromi Yamakawa.
\newblock Transport properties of polymer chains in dilute solution:
  hydrodynamic interaction.
\newblock {\em The Journal of Chemical Physics}, 53(1):436--443, 2003.

\bibitem{ying2009fast}
Lexing Ying.
\newblock Fast algorithms for boundary integral equations.
\newblock In {\em Multiscale Modeling and Simulation in Science}, pages
  139--193. Springer, 2009.

\bibitem{yip1986note}
EL~Yip.
\newblock A note on the stability of solving a rank-$p$ modification of a
  linear system by the {S}herman-{M}orrison-{W}oodbury formula.
\newblock {\em SIAM Journal on Scientific and Statistical Computing},
  7(2):507--513, 1986.

\bibitem{zhao2005adaptive}
Kezhong Zhao, Marinos~N Vouvakis, and J-F Lee.
\newblock The adaptive cross approximation algorithm for accelerated method of
  moments computations of emc problems.
\newblock {\em Electromagnetic Compatibility, IEEE Transactions on},
  47(4):763--773, 2005.

\end{thebibliography}

\end{document}